\documentclass{amsart}

\newtheorem{theorem}{Theorem}

\usepackage{multicol,enumerate,bussproofs,amssymb,txfonts}
\usepackage{xcolor}
\usepackage{stmaryrd}
\usepackage{natbib}

\usepackage[all,cmtip]{xy}

\theoremstyle{definition}
	\newtheorem{lemma}{Lemma}
	\newtheorem{cor}{Corollary}
	\newtheorem{ex}{Example}
	\newtheorem{mdef}{Definition}
	
	\newcommand{\LL}{\mathcal{L}}
	\newcommand{\At}{\mathsf{At}}
	\newcommand{\LRN}{\mathsf{LRN}}
        \newcommand{\lrcn}{\mathsf{lrcn}}
        \newcommand{\lpt}{\mathsf{lpt}}
        \newcommand{\LRCN}{\mathsf{LRCN}}
        \newcommand{\BM}{\mathbf{BM}}

        \newcommand{\CLV}{\mathbf{CLV}}

\newcommand{\logic}[1]{{\ensuremath {\mathbf{#1}}}}
\newcommand{\lrcnclose}[1]{#1^{\LRCN}}
\newcommand{\set}[1]{\{#1\}}

\title{Topics, Non-Uniform Substitutions, and Variable Sharing}

\author{Shawn Standefer}
\email{standefer@ntu.edu.tw}

\author{Shay Allen Logan}
\email{salogan@ksu.edu}

\author{Thomas Macaulay Ferguson}
\email{fergut5@rpi.edu}

\begin{document}
    \maketitle

\begin{abstract}
    The family of \emph{relevant logics} can be faceted by a hierarchy of increasingly fine-grained \emph{variable sharing properties}---requiring that in valid entailments $A\to B$, some atom must appear in both $A$ and $B$ with some additional condition (\emph{e.g.}, with the same sign or nested within the same number of conditionals). In this paper, we consider an incredibly strong variable sharing property of \emph{lericone relevance} that takes into account the path of negations and conditionals in which an atom appears in the parse trees of the antecedent and consequent. We show that this property of lericone relevance holds of the relevant logic $\mathbf{BM}$ (and that a related property of \emph{faithful} lericone relevance holds of $\mathbf{B}$) and characterize the largest fragments of classical logic with these properties. Along the way, we consider the consequences for lericone relevance for the theory of subject-matter, for Logan's notion of hyperformalism, and for the very definition of a relevant logic itself.
\end{abstract}

\section{Introduction}

A logic \logic{L} is said to enjoy variable sharing iff whenever $A\to B$ is a theorem of \logic{L}, then $A$ and $B$ share a propositional atom. An important early result in the history of relevant logic was Belnap's proof that sublogics of the relevant logic \logic{R} enjoyed variable sharing.\footnote{\citet{belnap1960}; for overviews of relevant logics, see \citet{Dunn2002-DUNRL} or \citet{Bimbo:2006aa}. 
For more in depth discussion, see \citet{anderson1975}, \citet{routley1982}, \citet{Read1988}, \citet{mares2004}, \citet{Logan:2023aa}, \citet{standefer202x} among others.
}

But variable sharing, it turns out, comes in many flavors. One example of this is \emph{strong variable sharing}---also due to Belnap---which requires that the shared atom $p$ have the same polarity in both $A$ and $B$. Another, due to \citet{brady1984}, is \emph{depth relevance}, which requires the shared atom to occur in the scope of the same number of conditionals in $A$ as in $B$.\footnote{For more on depth relevance, see \citet{robles2012, robles2014} and 
\citet{logan2022}.} One can combine these to get the strong depth relevance of \citet{logan2021}. Ferguson and Logan~\citeyearpar{FergusonForthcoming-FERTTA-7}, introduced and defended two yet more restrictive versions of variable sharing, which were shown to be enjoyed by the weak relevant logics $\mathbf{B}$ and $\mathbf{BM}$. We will present these logics in Section 2.

While variable sharing properties are of special interest to relevant logicians, they also arise in other areas of logic and philosophy. One such area is the theory of topic, where variable sharing properties have been discussed by, among others, Yablo \citeyearpar{yablo2014} and Berto \citeyearpar{berto2022}, respectively. But much work in this area has been governed either explicitly or implicitly by two principles of \emph{topic transparency} that can be described as follows:

\vspace{.2cm}

\noindent\emph{Negation Transparency:} For any sentence $\Phi$, the topics of $\Phi$ and $\neg\Phi$ are identical.

\vspace{.2cm}

\noindent\emph{Conditional Transparency:} For an intensional conditional $\rightarrow$ and sentences $\Phi,\Psi$ the topic of $\Phi\rightarrow\Psi$ is the fusion of the topic of $\Phi$ and the topic of $\Psi$

\vspace{.2cm}

\noindent In many quarters, these assumptions have been challenged. The deductive system $\mathsf{AC}$, introduced by Angell~\citeyearpar{angell1977}, implicitly violates Negation Transparency, something made clear by Fine's \citeyearpar{fine2015} analysis of subject-matter in $\mathsf{AC}$, and Angell~\citeyearpar[121]{angell1989} gives an argument to this effect. Meanwhile, Berto~\citeyearpar{berto2022} concedes  that topic transparency is questionable in intensional contexts---pushing against Conditional Transparency---an idea which is explored at length by Ferguson \citeyearpar{ferguson2023a}.

Ferguson and Logan \citeyearpar{FergusonForthcoming-FERTTA-7} argued that their novel forms of variable sharing were natural consequences of challenging these features. If one takes negation and the relevant conditional to be \emph{topic transformative}, then ensuring relevance---in the sense of topic overlap---must require a stronger standard than to merely ensure the mutual occurence of some atomic formula in antecedent and consequent positions. What is required instead is the atomic formula shared between antecedent and consequent in any provable conditional occur in each position nested within the same configuration of instances of $\neg$ and $\rightarrow$.

The variation of these properties to be examined in this work involves a simple insight concerning Conditional Transparency: \emph{the topic of a conditional is sensitive to chirality}, that is, the contribution to the topic of a conditional made by an immediate subformula is influenced by whether it appears in the antecedent or consequent position. This can be summarized as a thesis of:

\vspace{.2cm}

\noindent\emph{Chiral Transparency:} For an intensional conditional $\rightarrow$ and sentences $\Phi,\Psi$ the topic of $\Phi\rightarrow\Psi$ is the same as that of $\Psi\rightarrow\Phi$

\vspace{.2cm}

\noindent There is precedent for the rejection of such a thesis. With respect to its reflection in formal systems, Ferguson~\citeyearpar{ferguson2024} showed that a subsystem of Parry's analytic implication distinguishes the topics of formulas $\varphi\rightarrow\psi$ and $\psi\rightarrow\varphi$. On a philosophical level, we also may recall a Yablovian defense of an emphasis on chirality in assessing the overall topic of an intensional conditional. Yablo \citeyearpar{yablo2014} makes a succinct but compelling argument that chirality influences topic in the case of \emph{predicates}, noting that insofar as ``\textsc{Man Bites Dog}'' touches on a more \emph{interesting} topic than ``\textsc{Dog Bites Man},'' the topics of the two must be distinct. As argued in \emph{e.g.} \cite{ferguson2023a,ferguson2023c}, this Yablovian argument can be adapted to show that pairs of conditional sentences that are converses of one another might part ways with respect to topic as well. As a consequence, a view of topic especially sensitive to the influence of intensional connectives will require a very strong flavor of variable sharing.

In addition to variable sharing properties, relevant logics, and indeed most common logics, are closed under \emph{uniform} substitution of formulas for atoms.\footnote{Pun\v{c}och\'a\v{r}~\citeyearpar{Puncochar2023-PUNLFS} provides a good overview of different generalizations and restrictions of uniform substitution and examples of uniform substitution failing.} Just as there are strengthenings of variable sharing, one can strengthen closure under substitutions by allowing \emph{non-uniform} substitutions.\footnote{Apart from the depth substitutions of the following sentence, there are some examples of non-uniform substitutions and generalizations of uniform substitution in the literature. Humberstone~\citeyearpar{Humberstone2005-LLOFWO} discusses scattered substitution, and Ferguson~\citeyearpar{Ferguson2017-FERDQO} discusses another notion under that same name. Elsehwere, Humberstone~\citeyearpar{Humberstone:2005aa} discusses internal uniform substitution. 
} 
One such example is depth substitutions, studied by Logan~\citeyearpar{logan2021}, where the formula assigned to an (occurrence of an) atom can vary with the depth of that (occurrence of that) atom. Many relevant logics that enjoy the depth relevance property are, as Logan~\citeyearpar{logan2021} showed, also closed under depth substitutions, and so are \emph{hyperformal}, to use Logan's term.\footnote{In spite of the claims in Logan~\citeyearpar{logan2021}, the converse of this claim does not hold---see Logan~\citeyearpar{logan2023correction} for more.} Occurrences of the same atom at different depths are different atoms, from the point of view of the logic. 
Further developing the themes above, we will consider additional forms of non-uniform substitutions, giving rise to another form of hyperformality, and show that logics that are closed under these also enjoy related variable sharing properties. 

We will close this introduction with a brief overview of the paper. In the next section, we will briefly define the standard relevant logics of interest, $\logic{BM} $ and $\logic{B}$. Then, in Section \ref{Section:Lericones}, we will define the key concept of a lericone. We show that $\logic{BM} $ and $\logic{B}$ are closed under certain classes of lericone-sensitive substitutions, and, following this, we discuss the philosophical significance of our results, and variable sharing properties more generally. In Section \ref{Section:CLV}, we will study the largest sublogic of classical logic that is closed under lericone-sensitive substitutions, which we call $\logic{CLV}$. We prove that $\logic{CLV}$ enjoys variable sharing, using a new proof technique, along with a related system $\mathbf{CLV}^{\sim}$ closed under a more restricted class of substitutions. We show that the consequence relation for $\logic{CLV}$ is compact, and in Section \ref{Section:Tableaux}, we give tableau systems for $\logic{CLV}$ and $\mathbf{CLV}^{\sim}$, whose adequacy we prove. Finally, we close in Section \ref{Section:Conclusion} with some directions for future work.

\section{Setup}

We work in a standard propositional language $\LL$ formulated using the set of atomic formulas $\At=\{p_i\}_{i=1}^\infty$ and the connectives $\land$, $\lor$, $\to$, and $\neg$ each with their expected arities. We identify (as a matter of convenience and \textit{not} as a matter of philosophical conviction) logics with their sets of theorems. There are two logics we will be concerned with throughout: $\mathbf{BM}$ and $\mathbf{B}$.

The logic $\mathbf{BM}$  received its most extensive examination by Fuhrmann~\citeyearpar{fuhrmann1988} and was discussed by 
Priest and Sylvan \citeyearpar{priest1992a}.\footnote{We note that Fuhrmann~\citeyearpar[27]{fuhrmann1988} says that the then-planned sequel to \emph{Relevant Logics and Their Rivals} would take $\mathbf{BM}$ as the basic relevant logic, rather than $\mathbf{B}$, which indicates that $\mathbf{BM}$ was known and studied by relevant logicians of the time.} 
It has received a spate of recent interest, being mentioned not only by Ferguson and Logan \citeyearpar{FergusonForthcoming-FERTTA-7}, but also by Standefer \citeyearpar{standefer2019tracking, standefer2023weak}, Tedder and Bilkov\'a \citeyearpar{tedder2022relevant}, and Ferenz and Tedder \citeyearpar{ferenz2023neighbourhood}. It is axiomatized as follows: 
\begin{enumerate}[({A}1)]
    \item $A\to A$
    \item $A\land B\to A,A\land B\to B$
    \item $A\to A\lor B, B\to A\lor B$
    \item $A\land(B\lor C)\to(A\land B)\lor(A\land C)$
    \item $((A\to B)\land(A\to C))\to(A\to(B\land C))$
    \item $((A\to C)\land(B\to C))\to((A\lor B)\to C)$
    \item $\neg(A\land B)\to(\neg A\lor\neg B)$
    \item $(\neg A\land\neg B)\to\neg(A\lor B)$
\end{enumerate}
\begin{enumerate}[({R}1)]    
    \item $A,B\Rightarrow A\land B$
    \item $A,A\to B\Rightarrow B$
    \item $A\to B\Rightarrow\neg B\to\neg A$
    \item $A\to B,C\to D\Rightarrow(B\to C)\to(A\to D)$
\end{enumerate}

The logic $\mathbf{B}$ is $\BM$ with one new axiom and one new rule:
\begin{itemize}
    \item[(A9)] $\neg\neg A\to A$
    \item[(R5)] $A\to\neg B\Rightarrow B\to\neg A$
\end{itemize}
Note that with these additions, Axioms A7 and A8 as well as rule R3 are superfluous. $\mathbf{B}$ was originally thought of, by Routley et al.~\citeyearpar{routley1982}, as something like a basement level relevant logic. It has received a good deal of philosophical attention over the years, and plays an important role in both \citet{Read1988} and \citet{Logan:2023aa}.

\section{Lericones}\label{Section:Lericones}

\begin{mdef}
    We define the set of lericone sequences, $\LRCN$, as follows:
    \begin{itemize}
        \item The empty sequence, $\varepsilon$, is a lericone sequence.
        \item The one-member sequence $c$ is a lericone sequence.
        \item If $\overline{x}$ is a lericone sequence, so are $l\overline{x}$, $r\overline{x}$, and $n\overline{x}$.
        \item Nothing else is a lericone sequence.
    \end{itemize}
\end{mdef}

We let $\LRN$ be the set subset of all $c$-free members of $\LRCN$---equivalently, $\LRN$ is the set of all finite sequences from $\{l,r,n\}$. The word `lericone' is a mnemonic portmanteau: it stands for `LEft, RIght, COnditional, NEgation'. This connection is explained by the next definition:

\begin{mdef}
    The set of partial lericone parsing trees of a formula $A$ is defined as follows:
    \begin{itemize}\setlength\itemsep{1ex}
        \item $\underset{\varepsilon}{A}$ is a partial lericone parsing tree for $A$.
        \item Given a (possibly empty) tree $T$ and $*\in\{\land,\lor\}$, if 
        \xymatrix@=1mm{
            T\ar[d]\\ 
            \underset{\overline{x}}{B*C}
        } 
        is a partial lericone parsing tree for $A$, then so is  
        \xymatrix@=1mm{
            & T\ar[d] & \\ 
            & \underset{\overline{x}}{B*C}\ar[dr]\ar[dl] & \\
            \underset{\overline{x}}{B} & & \underset{\overline{x}}{C}
        }
        \item Given a (possibly empty) tree $T$, if 
        \xymatrix@=1mm{
            T\ar[d]\\ 
            \underset{\varepsilon}{B\to C}
        } 
        is a partial lericone parsing tree for $A$, then so is  
        \xymatrix@=1mm{
            & T\ar[d] & \\ 
            & \underset{\varepsilon}{B\to C}\ar[dr]\ar[dl] & \\
            \underset{c}{B} & & \underset{c}{C}
        }
        \item Given a (possibly empty) tree $T$ and $\overline{x}\neq\varepsilon$, if 
        \xymatrix@=1mm{
            T\ar[d]\\ 
            \underset{\overline{x}}{B\to C}
        } 
        is a partial lericone parsing tree for $A$, then so is  
        \xymatrix@=1mm{
            & T\ar[d] & \\ 
            & \underset{\overline{x}}{B\to C}\ar[dr]\ar[dl] & \\
            \underset{l\overline{x}}{B} & & \underset{r\overline{x}}{C}
        }
        \item Given a (possibly empty) tree $T$, if 
        \xymatrix@=1mm{
            T\ar[d]\\ 
            \underset{\overline{x}}{\neg B}
        } 
        is a partial lericone parsing tree for $A$, then so is  
        \xymatrix@=1mm{
            T\ar[d]\\ 
            \underset{\overline{x}}{\neg B}\ar[d] \\
            \underset{n\overline{x}}{B}
        } 
    \end{itemize}
\end{mdef}

\begin{mdef}
    \textit{The lericone parsing tree} for $A$, $\lpt(A)$ is the maximal partial lericone parsing tree for $A$. 
\end{mdef}

\begin{ex}
    Each of the following is a lericone parsing tree for $\neg p\to(p\to q)$; the rightmost one is $\lpt(\neg p\to(p\to q))$:
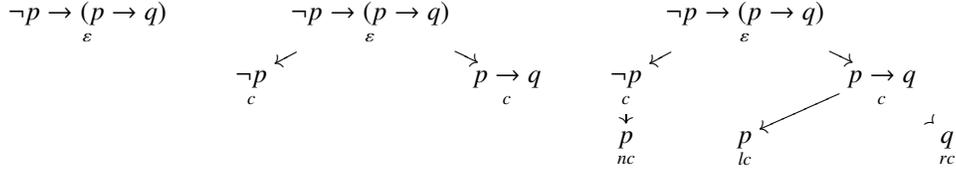
    \begin{displaymath}
        \xymatrix@=1mm{
            \underset{\varepsilon}{\neg p\to(p\to q)}
        }
        \qquad
        \xymatrix@=1mm{
            &   \underset{\varepsilon}{\neg p\to(p\to q)}\ar[dr]\ar[dl]    & \\
            \underset{c}{\neg p} &   &\underset{c}{p\to q}
        }
        \qquad
        \xymatrix@=1mm{
            &   \underset{\varepsilon}{\neg p\to(p\to q)}\ar[dr]\ar[dl]   & &\\
            \underset{c}{\neg p}\ar[d] &   &\underset{c}{p\to q}\ar[dr]\ar[dl] & \\
            \underset{nc}{p} & \underset{lc}{p} & & \underset{rc}{q}
        }
    \end{displaymath}
\end{ex}

Note that the sequences occurring under the formulas in $\lpt(A)$ implicitly define a function mapping each occurrence $A[B]$ of $B$ as a subformula of $A$ to a lericone sequence $\lrcn(A[B])$.\footnote{We note that $\lrcn(A[B])$ only makes sense when $B$ occurs in $A$. Our usage is implicitly existential, since we are assuming that the required occurrence is in the formula. 
}

\begin{ex}
    If we underline the occurrence we're interested in, we can list the value of the $\lrcn$ function at each atom-occurrence in the preceding example as follows:
    \begin{itemize}
        \item $\lrcn(\neg\underline{p}\to(p\to q))=nc$
        \item $\lrcn(\neg p\to(\underline{p}\to q))=lc$
        \item $\lrcn(\neg p\to(p\to\underline{q}))=rc$
    \end{itemize}

    A thing to note is that a lericone sequence is always only assigned to a subformula occurrence. So, as an example, while $\lrcn(\underline{p}\to q)=c$, $\lrcn(\neg(\underline{p}\to q))=ln$.
\end{ex}

\begin{mdef}
    For $\overline{x}\in\LRN$, we define the $c$-transform of $\overline{x}$, $t(\overline{x})$ as follows:
    \begin{displaymath}
        t(\overline{x})=
            \left\{\begin{array}{rl}
                \overline{y}c & \text{ if }\overline{x}=\overline{y}l\text{ or }\overline{x}=\overline{y}r \\
                \overline{x} & \text{ otherwise}
            \end{array}\right.
    \end{displaymath}
\end{mdef}

\begin{ex}\label{t_example}
    Write $A[B]\to C$ to mean that $B$ occurs in $A\to C$ as a subformula of $A$, and similarly for $C\to A[B]$, which we will distinguish from $(C\to A)[B]$, where we do not take $B$'s location to be specified. Here's an observation: if $\lrcn(A[B]\to C)=\overline{x}c$ or $\lrcn(C\to A[B])=\overline{x}c$, then $\lrcn(A[B])=t(\overline{x})$. We won't prove this---though it's a good exercise and can be proved by a straightforward induction on $A$---and will instead examine a few examples. 
    \begin{description}
        \item[Example \ref{t_example}a] Here we take $A[B]$ to be $\underline{p}\to p$ and $C$ to be $p$. Now note that $\lrcn((\underline{p}\to p)\to p)=lc$ while $\lrcn(\underline{p}\to p)=c$. Since $c=t(l)$, this confirms the claim.
        \item[Example \ref{t_example}b] Here we take $A[B]$ to be $\neg(p\to\underline{p})$ and $C$ to be $q$. Now note that $\lrcn(q\to\neg(p\to\underline{p}))=rnc$ while $\lrcn(\neg(p\to\underline{p}))=rn$. Since $t(rn)=rn$, this confirms the claim.
    \end{description}
\end{ex}

We leave the proof of the following fact about $c$-transforms to the reader:

\begin{lemma}\mbox{}
    \begin{itemize}
        \item 
        $t(l\overline{x})=\left\{\begin{array}{rl}
        c & \text{ if }\overline{x}=\varepsilon \\
        lt(\overline{x}) & \text{ otherwise }
        \end{array}\right.$ 
        \item 
        $t(r\overline{x})=\left\{\begin{array}{rl}
        c & \text{ if }\overline{x}=\varepsilon \\
        rt(\overline{x}) & \text{ otherwise }
        \end{array}\right.$
        \item $t(n\overline{x})=nt(\overline{x})$.
    \end{itemize}
\end{lemma}

Much of our interest in this paper concerns substitutions that are sensitive to lericones. 
To help settle notation and intuitions, we will first define (plain) substitutions.
\begin{mdef}[Plain substitutions]
A \emph{plain substitution} is a function $\At\to\LL$. Given a plain substitution $\sigma$, we extend it to a function $\LL\to\LL$ (which we will also call $\sigma$) as follows:
\begin{itemize}
    \item $\sigma(A\land B)=\sigma(A)\land\sigma(B)$.
    \item $\sigma(A\lor B)=\sigma(A)\lor\sigma(B)$.
    \item $\sigma(\neg A)=\neg\sigma(A)$.
    \item $\sigma(A\to B)=\sigma(A)\to\sigma(B)$
\end{itemize}
\end{mdef}
\begin{mdef}[Lericone substitutions]
A \emph{lericone substitution} is a function $\LRCN\times\At\to\LL$. Given a lericone substitution $\sigma$ we extend it to a function $\LRCN\times\LL\to\LL$ (which we also call $\sigma$) as follows:
\begin{itemize}
    \item $\sigma(\overline{x},A\land B)=\sigma(\overline{x},A)\land\sigma(\overline{x},B)$.
    \item $\sigma(\overline{x},A\lor B)=\sigma(\overline{x},A)\lor\sigma(\overline{x},B)$.
    \item $\sigma(\overline{x},\neg A)=\neg\sigma(n\overline{x},A)$.
    \item $\sigma(\overline{x},A\to B)=\left\{
        \begin{array}{rl}
            \sigma(c,A)\to\sigma(c,B) & \text{ if }\overline{x}=\varepsilon \\
            \sigma(l\overline{x},A)\to\sigma(r\overline{x},B) & \text{ otherwise}.
        \end{array}\right.$
\end{itemize}
\end{mdef}
\noindent We can see plain substitutions as a special class of lericone substitutions, namely those in which $\sigma(\overline{x}, A)=\sigma(\overline{y}, A)$, for all $\overline{x},\overline{y}\in\LRCN$.

We will adapt an important definition from \citet{Leach-Krouse2024-LEALIT-4}, that of atomic injective substitutions. As we will see below, atomic injective substitutions play a crucial role in our results.
\begin{mdef}
    We say that a lericone substitution is \emph{atomic} when its range is a subset of $\At$. We say that a lericone substitution is \emph{atomic injective} when it is atomic and injective as a two-place function.
\end{mdef}
\noindent
It's important to note that all it means when we say a given lericone substitution is atomic is that it maps each member of $\LRCN\times\At$ to a member of $\At$. It most definitely does \textit{not} mean that it(s extension) maps each member of $\LRCN\times\LL$ to a member of $\At$---indeed, a moment's reflection makes it clear that no lericone substitution does that. We also emphasize that, for a given atomic injective lericone substitution $\iota$, if $\overline{x}\neq\overline{y}$, then $\iota(\overline{x},p)\neq \iota(\overline{y}, p)$.

It will be useful on occasion in what follows to have a specific atomic injective on hand. So we construct one via what is essentially G\"odel coding as follows. First, let $g(l)=1$, $g(r)=2$, $g(c)=3$, and $g(n)=4$. Let $\pi_i$ be the $i$th prime. For $\overline{x}=x_1\dots x_n\in\LRCN$, let $g(\overline{x})=\prod_{i=2}^{n+1}\pi_i^{g(x_{i-1})}$. Finally, define $g(\overline{x},p_i)=p_{2^ig(\overline{x})}$. We take it to be clear that $g$ is atomic injective.

\begin{ex}
Consider the formula $\neg p_1\to(p_1\to p_1)$. The lericone sequences for the three occurrences of $p_1$, from left to right, are $nc$, $lc$, and $rc$, respectively. Consequently, 
\[
g(\neg p_1\to(p_1\to p_1))=\neg p_{20250}\to(p_{750}\to p_{2250}),
\]
with the result that the three occurrences of $p_1$ are mapped to different atoms.
\end{ex}

\begin{lemma}\label{transform_lemma}
    If $\sigma$ and $\tau$ are lericone substitutions and for all $p\in\At$ and all $\overline{x}\in\LRN$ we have that $\tau(\overline{x}c,p)=\sigma(t(\overline{x}),p)$, then in fact for all $A\in\LL$ and all $\overline{x}\in\LRN$ we have that $\tau(\overline{x}c,A)=\sigma(t(\overline{x}),A)$.
\end{lemma}
\begin{proof}
    By induction on $A$. The hypothesis gives the base case. The inductive cases for conjunction and disjunction are immediate from IH. 

    For negations note that $\tau(\overline{x}c,\neg A)=\neg\tau(n\overline{x}c,A)$. But by IH, this is $\neg\sigma(t(n\overline{x}),A)=\neg\sigma(nt(\overline{x}),A)=\sigma(t(\overline{x}),\neg A)$.

    For conditionals first note that $\tau(\overline{x}c,A\to B)=\tau(l\overline{x}c,A)\to\tau(r\overline{x}c,B)$, and by IH this is $\sigma(t(l\overline{x}),A)\to\sigma(t(r\overline{x}),B)$. We now consider two cases.

    If $\overline{x}=\varepsilon$, then by the above equality, $\tau(c,A\to B)=\sigma(t(l),A)\to\sigma(t(r), B)$. But $t(l)=t(r)=c$, so this becomes $\sigma(c,A)\to\sigma(c,B)=\sigma(\varepsilon,A\to B)=\sigma(t(\varepsilon),A\to B)$ as required.

    If $\overline{x}\neq\varepsilon$, then by the above equality, $\tau(\overline{x}c,A\to B)=\sigma(t(l\overline{x}),A)\to\sigma(t(r\overline{x}),B)$. By Lemma~\ref{transform_lemma}, $t(l\overline{x})=lt(\overline{x})$ and $t(r\overline{x})=rt(\overline{x})$. So this becomes $\sigma(lt(\overline{x}),A)\to\sigma(rt(\overline{x}),B)=\sigma(t(\overline{x}),A\to B)$ as required.
\end{proof}

\begin{mdef}
    If $\sigma$ is a lericone substitution, then we define the lericone substitution $t(\sigma)$ as follows:
    \begin{displaymath}
        t(\sigma)(\overline{y},p) =\left\{
            \begin{array}{lr}
                \sigma(t(\overline{x}),p) & \text{ if }\overline{y}=\overline{x}c \\
                p & \text{otherwise}
            \end{array}\right.
    \end{displaymath}
\end{mdef}

\begin{cor}[Corollary to Lemma~\ref{transform_lemma}]
    For all $\overline{x}\in\LRN$ and all $A\in\LL$, $t(\sigma)(\overline{x}c,A)=\sigma(t(\overline{x}),A)$.
\end{cor}
\noindent Note that in the definition of $t(\sigma)$, nothing important is happening with the `otherwise' clause and in fact this can be chosen arbitrarily with no changes in what follows. Note also that $t(\sigma)$ is, in effect, a lericone substitution that simply treats each sequence as though its terminal `$c$', if it has one, isn't there. So it, in a certain sense, ignores parts of the sequence it's being handed. We will also need to do the opposite, and \textit{add in} an additional sequence sometimes. To that end, we have the following lemma whose proof is sufficiently like the proof of Lemma~\ref{transform_lemma} that we leave it to the reader:

\begin{lemma}\label{adding}
    Let $\sigma$ and $\tau$ be lericone substitutions, $\overline{y}\in\LRN$, and let $\tau(\overline{x}c,p)=\sigma(\overline{xy}c,p)$ for all $p\in\At$ and $\overline{x}\in\LRN$. Then for all $\overline{x}\in\LRN$ and all $A$, $\tau(\overline{x}c,A)=\sigma(\overline{xy}c,A)$.
\end{lemma}

\begin{mdef}
    If $\sigma$ is a lericone substitution and $\overline{y}\in\LRN$, then we define the lericone substitution $\sigma^{\overline{y}}$ as follows:
    \begin{displaymath}
        \sigma^{\overline{y}}(\overline{z},p) =\left\{
            \begin{array}{lr}
                \sigma(\overline{xy}c,p) & \text{ if }\overline{z}=\overline{x}c \\
                p & \text{otherwise}
            \end{array}\right.
    \end{displaymath}
\end{mdef}

\begin{cor}[Corollary to Lemma~\ref{adding}]
    For all $\overline{x}\in\LRN$ and all $A\in\LL$, $\sigma^{\overline{y}}(\overline{x}c,A)=\sigma(\overline{xy}c,A)$.
\end{cor}

\begin{theorem}\label{BM_invariance}
    If $A'$ is a theorem of $\BM$, then for all lericone substitutions $\sigma$, $\sigma(\varepsilon,A')$ is a theorem of $\BM$ as well.
\end{theorem}
\begin{proof}
    By induction on the derivation of $A'$.\footnote{We've used the metavariable $A'$ here because the metavariables $A$, $B$, $C$, and $D$ are already in use, $E$ (in the form of $\mathbf{E}$) is mostly already taken, $F$ (in the form of $\mathsf{F}$) is too, and none of us could remember what the next letter after that was.} One can quickly verify by inspection that for each axiom, applying a lericone substitution returns an instance of the same axiom. We will look at two examples to illustrate. To see that applying lericone substitutions to (A7) results in more instances of (A7), note that the $\LRCN$-sequence for each displayed $A$ and $B$ is $nc$, so $\sigma(\varepsilon, \neg(A\land B)\to(\neg A\lor\neg B))=\neg(\sigma(nc, A)\land \sigma(nc, B))\to(\neg\sigma(nc, A)\lor\neg\sigma(nc, B))$, which is an instance of (A7). Similarly, for (A5), note that the lericone sequences for the displayed occurrences of $A$ are all $lc$ and the lericone sequences for the displayed occurrences of $B$ and $C$ are all $rc$. By reasoning similar to that for (A7), the result of applying a lericone substitution to (A5) is another instance of (A5).

    The case where the last rule applied in the proof was R1 is immediate. Suppose the last rule was R2, and let $\sigma$ be a lericone substitution. By IH applied to $A$, $\sigma(\varepsilon,A)$ is a theorem of $\BM$.\footnote{Throughout, `IH' abbreviates `the inductive hypothesis'. We will seldom take the time to enunciate exactly what IH is, since it will always be clear from context.} By IH applied to $A\to B$, $t(\sigma)(\varepsilon,A\to B)$ is a theorem of $\BM$. So $t(\sigma)(c,A)\to t(\sigma)(c,B)$ is a theorem of $\BM$. But by Lemma~\ref{transform_lemma}, $t(\sigma)(c,A)=\sigma(\varepsilon,A)$ and $t(\sigma)(c,B)=\sigma(\varepsilon,B)$. So $\sigma(\varepsilon,A)\to\sigma(\varepsilon,B)$ is a theorem of $\BM$. Thus, $\sigma(\varepsilon,B)$ is a theorem of $\BM$, as required. 

    Suppose the last rule applied was R3 and let $\sigma$ be a lericone substitution. By IH, $\sigma^n(\varepsilon,A\to B)$ is a theorem of $\BM$. Thus $\sigma^n(c,A)\to\sigma^n(c,B)$ is a theorem of $\BM$. But then by Lemma~\ref{adding}, $\sigma(nc,A)\to\sigma(nc,B)$ is a theorem of $\BM$, whence so also is $\neg\sigma(nc,B)\to\neg\sigma(nc,A)=\sigma(\varepsilon,\neg B\to\neg A)$ as required. 

    Suppose the last rule applied was R4 and let $\sigma$ be a lericone substitution. By IH, both $\sigma^l(\varepsilon,A\to B)$ and $\sigma^r(\varepsilon,C\to D)$ are theorems of $\BM$. Thus $\sigma^l(c,A)\to\sigma^l(c,B)=\sigma(lc,A)\to\sigma(lc,B)$ and $\sigma^r(c,C)\to\sigma^r(c,D)=\sigma(rc,C)\to\sigma(rc,D)$ are theorems. But then so is the following:
    \begin{displaymath}
        (\sigma(lc,B)\to\sigma(rc,C))\to(\sigma(lc,A)\to\sigma(rc,D))
    \end{displaymath}
    But this just is $\sigma(\varepsilon,((B\to C)\to(A\to D))$.
\end{proof}

\begin{cor}\label{BM_var_shar}
    If $A\to B$ is a theorem of $\BM$, then there is a variable $p$, an $\overline{x}\in\LRN$, and occurrences $A[p]$ of $p$ in $A$ and $B[p]$ of $p$ in $B$ so that $\lrcn(A[p]\to B)=\lrcn(A\to B[p])=\overline{x}c$. Thus $p$ occurs under the same lrn-sequence in both $A$ and $B$. 
\end{cor}
\begin{proof}
    Let $A\to B$ be a theorem of $\BM$ and $g$ be the G\"odel substitution defined above. By Theorem~\ref{BM_invariance}, $g(\varepsilon,A\to B)$ is a theorem of $\BM$. So $g(\varepsilon,A\to B)=g(c,A)\to g(c,B)$ is also a theorem of $\mathbf{R}$. So some variable---let's say $p_{2^ig(\overline{x})}$ occurs simultaneously in both $g(c,A)$ and $g(c,B)$. It follows that $p_i$ occurs under the $\LRN$-sequence $\overline{x}$ in both $A$ and $B$ and thus that $\lrcn(A[p]\to B)=\lrcn(A\to B[p])=\overline{x}c$.
\end{proof}

We note here that this proof of Corollary~\ref{BM_var_shar} is parasitic on the existing variable sharing result for $\mathbf{R}$. We will show below that this parasitism is not essential.

\subsection{$\logic{B}$ is Closed Under Faithful Lericone Substitutions}

$\BM$, as mentioned, is ignominiously located in the subbasement of the relevant world. Here we will show that a very mild modification of the above result lets us push the result from the subbasement to the basement proper.

\begin{mdef}
    Define $\sim'$ to be the relation containing all and only pairs of the form $\langle \overline{x}nn\overline{y}, \overline{xy} \rangle$. Let $\sim$ be the equivalence relation generated by $\sim'$, $\underline{LRCN}$ be the set of equivalence classes of $\LRCN$ under $\sim$, and $\underline{LRN}$ be the set of equivalence class of $\LRN$ under $\sim$. We say that a lericone substitution $\sigma$ is \textit{faithful} when $\sigma(\overline{x},A)=\sigma(\overline{y},A)$ whenever $\overline{x}\sim\overline{y}$. 
\end{mdef}

\begin{theorem}
    If $A'$ is a theorem of $\mathbf{B}$, then for all faithful lericone substitutions $\sigma$, $\sigma(\varepsilon,A')$ is a theorem of $\mathbf{B}$ as well.
\end{theorem}
\begin{proof}
    By induction on the proof of $A'$. For axioms and rules in $\BM$, the result follows from Theorem~\ref{BM_invariance}. We're thus left to check the two new additions.

    For the new axiom, we let $\sigma$ be a faithful lericone substitution and compute as follows:
    \begin{align*}
        \sigma(\varepsilon,\neg\neg A\to A) & = \sigma(c,\neg\neg A)\to\sigma(c,A) \\
         & = \neg\neg\sigma(nnc,A)\to\sigma(c,A)
    \end{align*}
    But since $\sigma$ is faithful, $\sigma(nnc,A)=\sigma(c,A)$. Thus the final line of the computation again records an instance of a $\mathbf{B}$-axiom.

    Now suppose the last rule used in the proof was R5. Then $A'=D\to\neg C$, and the last step in the proof concluded this from $C\to\neg D$. Let $\sigma$ be a faithful lericone substitution. By IH, $\sigma^n(\varepsilon,C\to\neg D)$ is a theorem of $\mathbf{B}$ so $\sigma^n(c,C)\to\neg\sigma^n(nc,D)$ is a theorem of $\mathbf{B}$. Thus $\sigma(nc,C)\to\neg\sigma(nnc,D)$ is a theorem of $\mathbf{B}$. So by an application of R5, $\sigma(nnc,D)\to\neg\sigma(nc,C)$ is a theorem of $\mathbf{B}$. Since $\sigma$ is faithful, $\sigma(nnc,D)=\sigma(c,D)$. Thus $\sigma(c,D)\to\neg\sigma(nc,C)=\sigma(\varepsilon,D\to\neg C)$ is a theorem of $\mathbf{B}$ as required.
\end{proof}

We end by stating without proof the corresponding variable sharing result.
\begin{cor}\label{B_var_shar}
    If $A\to B$ is a theorem of $\mathbf{B}$, then there is a variable $p$, an $\overline{x}\in\underline{LRN}$, and occurrences $A[p]$ of $p$ in $A$ and $B[p]$ of $p$ in $B$ so that $\lrcn(A[p]\to B)\sim\lrcn(A\to B[p])\sim\overline{x}c$. Thus $p$ occurs under equivalent $\LRN$-sequences in both $A$ and $B$. 
\end{cor}

\subsection{Philosophical Reflections}

Variable sharing is multifarious. The above results are the seventh and eighth such results of which the authors are aware. It's natural to wonder whether the proliferation is justified. We think it is, and will review a handful of reasons for thinking this to be so.

For one, variable sharing was historically taken to be a necessary but not sufficient condition on being a relevant logic.\footnote{ See, for example, \citet{sep-logic-relevance}.} 
Standefer \citeyearpar{standefer2024} has proposed using variable sharing, with one other plausible condition, to define the class of relevant logics. If we take this on board, different types of variable sharing would seem to give us different flavors of relevance. 

In a bit more detail, \citet{standefer2024} proposes a three-component definition for the class of relevant logics. The first component is the definition of a \emph{logic}: a logic is a set of formulas closed under plain substitutions. The second component is a definition of the \emph{relevant} portion: a logic \logic{L} is a \emph{proto-relevant} logic iff it satisfies the variable sharing criterion, namely that if $A\to B\in\logic{L}$, then $A$ and $B$ share an atom. For the third component, Standefer defines the relevant \emph{logics} to be the proto-relevant logics that are closed under modus ponens, (R1) above, and adjunction, (R2) above.

By the this definition all the standard relevant logics---including the ones that Anderson and Belnap discussed, such as \logic{R} and \logic{E}---count as relevant logics. But so do many other logics not typically included in the relevant family. For example, linear logic, many connexive logics, and even some non-transitive logics are included in the class, while classical and intuitionistic logic are excluded. Importantly, however, the definition includes logics that are properly stronger than \logic{R}, as well as some that are incomparable with \logic{R}, such as \logic{TMingle}.\footnote{See \citet{Mendez2012-MNDTEP}.}
In fact, below we will isolate another relevant logic incomparable with \logic{R}, as well as with \logic{TMingle}.

Towards the end of his paper, Standefer considers that we might isolate different flavors of relevance by modifying the second component of his definition so as to require some stronger sort of variable sharing. Thus, for example, we might call a set of formulas a \emph{lericone-relevant logic} if it is closed under plain substitutions, satisfies the $\LRCN$-variable sharing criterion, and is closed under modus ponens and adjunction. But what we see here is that there is an alternative way we might introduce new flavors of relevance: rather than strengthen in the second component we can strengthen in the first. 

In more detail, while there is nothing wrong with saying that (\emph{e.g.}) a lericone-relevant logic is a set of formulas that is closed under plain substitutions, satisfies the $\LRCN$-variable sharing criterion, and is closed under modus ponens and adjunction, we might alternatively say that a \emph{relevant lericone-hyperformal logic} is a set of formulas that is is closed under $\LRCN$-substitutions, satisfies the ordinary variable-sharing criterion, and is closed under modus ponens and adjunction. As we will see, it turns out that all relevant lericone-hyperformal logics that are sublogics of classical logic, \logic{CL}, are lericone-relevant logics.

The upshot of this is a general way to identify classes of relevant logics via different senses of formal relevance. 
The logics obtained via the use of lericone substitutions are well-behaved, albeit weak. Valid implications in these logics ensure a very tight connection between antecedent and consequent, much tighter than basic variable sharing, or even depth variable sharing. 

Second, we can turn our attention to Anderson and Belnap's \citeyearpar[p.~33]{anderson1975} explication of relevance as ``common meaning content''. As one intuitive way of failing to share common content is a failure to share \emph{topic}, these increasingly fine-grained notions of variable sharing line up very naturally with the burgeoning theory of topic as investigated by \emph{e.g.} Yablo~\citeyearpar{yablo2014} and Berto~\citeyearpar{berto2022}. As discussed by Ferguson~\citeyearpar{ferguson2024} and Ferguson and Kadlecikova~\citeyearpar{fergusonkadlecikova2024}, lericone relevance reflects an acknowledgement that the topic of a subsentence is affected by features of the intensional connectives in which it is nested and that its topic-theoretic contribution to the complex in which it appears must be evaluated \emph{in situ}.

In $\mathbf{B}$, for example, the formula $(A\rightarrow B)\rightarrow(\neg B\rightarrow\neg A)$ is not provable. In the presence of rule (R3), the failure of the entailment is not a matter of truth conditions, but ought rather to be attributed to a lack of common topic between $A\rightarrow B$ and $\neg B\rightarrow\neg A$. If one accepts the arguments of Ferguson~\citeyearpar{ferguson2023d,ferguson2023e} suggesting that contraposition is not a topic-preserving operation, then this would recommend the topic theory implicit in $\mathbf{B}$.

Likewise, the failure of theoremhood of $\neg\neg A\rightarrow A$ in $\mathbf{BM}$ signals the failure of a thesis concerning the invariance of topic under applications of double negation, namely, the failure of the following:

\vspace{.2cm}

\noindent\emph{Involutive Transparency:} For any sentence $\Phi$, the topics of $\Phi$ and $\neg\neg\Phi$ are identical.

\vspace{.2cm}

\noindent $\mathbf{BM}$, in failing to identify the propositions $A$ and $\neg\neg A$ as having common content, appears to part ways with $\mathbf{B}$ on grounds of involutive transparency in this sense.

Such correspondences in weak relevant logics suggest the value of a very fine-grained and uniform hierarchy of variable sharing properties, as these correspondences align particular classes of relevant logics with particular linguistic views concerning topic transparency. If one's preferred theory of topic accepts Involutive and Chiral Transparency but rejects Conditional Transparency, (e.g.) the results of \cite{logan2021} suggest that a notion of relevance tailored to this theory of topic might be exhibited by the class of strong depth relevant logics. As the theory of topic matures---and determinations are made concerning various transparency theses---a panoply of variable sharing properties will help justify the selection of individual relevant logics.

Finally, as we saw in the introduction, there are good reasons for demanding something like lericone-hyperformality. What we will see below in Section~\ref{sec:CLV_is_relevant} is that lericone-hyperformality is almost enough on its own to ensure a logic enjoys ordinary variable sharing. So in fact, not only is hyperformality a novel path to new flavors of relevance, it's actually a novel path to relevance full stop.

\section{Classical Lericone Invariant Logic}\label{Section:CLV}

We have introduced and motivated the concepts of a lericone and lericone-sensitive substitutions. The natural next question is what logic one gets from these concepts. In this section, we will formulate this question in a precise way, in terms of closure under lericone-sensitive substitutions. We will define a class of assignments that are, likewise, sensitive to lericones and show that the lericone-invariant fragment of \logic{CL} coincides with the set of formulas valid with respect to these assignments. Following this, we will show that the resulting logic---a logic we call $\CLV$---enjoys variable sharing, in fact a new form of variable sharing, and so this logic is a new relevant logic. Its consequence relation is, additionally, compact. While we do not have a Hilbert axiomatization of this logic, we will provide a simple tableau system for it, which will be shown sound and complete. 

\subsection{Introducing and Situating the Logic}

A \emph{lericone-sensitive truth-value assignment} (from here on just an \textit{assignment}) is a function $f:\LRCN\times\At\longrightarrow\{0,1\}$. We extend such an assignment to a function (which we also call $f$) $\LRCN\times\LL\longrightarrow\{0,1\}$ as follows:
\begin{itemize}
    \item $f(\overline{x},A\land B)=\inf(f(\overline{x},A), f(\overline{x},B))$
    \item $f(\overline{x},A\lor B)=\sup(f(\overline{x},A), f(\overline{x},B))$
    \item $f(\overline{x},\neg A)=1-f(n\overline{x},A)$.
    \item $f(\overline{x},A\to B)=\left\{
        \begin{array}{rl}
            \sup(1-f(c\overline{x},A),f(c\overline{x},B)) & \text{ if $\overline{x}$ is $c$-free}\\
            \sup(1-f(l\overline{x},A),f(r\overline{x},B)) & \text{ otherwise}
        \end{array}
        \right.$
\end{itemize}

Given an assignment $f$, a set of sentences $X$ and a sentence $A$, we say that $f\vDash X\Yright A$ just if either $f(\varepsilon,B)=0$ for some $B\in X$ or $f(\varepsilon,A)=1$. We say that $X\Yright A$ is valid and write $X\vDash_{\LRCN} A$ just if $f\vDash X\Yright A$ for all assignments $f$. We say that $A$ is \emph{classically lericone valid} when $\emptyset\vDash_{\LRCN} A$, and we write $\CLV$ for the set of classically lericone valid formulas. 

\begin{mdef}
    For any set of formulas $X$, say that $X$ is \emph{closed under lericone substitutions} when for all lericone substitutions $\sigma$, if $A\in X$, then $\sigma(\varepsilon,A)\in X$.  
\end{mdef}

\begin{mdef}
    $\lrcnclose{X}=\{A:\sigma(\varepsilon,A)\in X\text{ for all lericone substitutions }\sigma\}$
\end{mdef}

\begin{mdef}
    Let $\sigma$ and $\tau$ be lericone substitutions. Define $(\sigma\star\tau)$ to be the lericone substitution $\langle\overline{x},p\rangle\mapsto\sigma(\overline{x},\tau(\overline{x},p))$. 
\end{mdef}

\begin{lemma}\label{lemma:star}
    $(\sigma\star\tau)(\overline{x},A)=\sigma(\overline{x},\tau(\overline{x},A))$ for all formulas $A$.
\end{lemma}
\begin{proof}
    By induction on $A$. If $A$ is an atom, the result is immediate from the definition of $(\sigma\star\tau)$. We examine one of the two conditional cases and leave the remainder of the cases to the reader.

    So suppose $\overline{x}\neq\varepsilon$. Then $(\sigma\star\tau)(\overline{x},A\to B)=(\sigma\star\tau)(l\overline{x},A)\to(\sigma\star\tau)(r\overline{x},B)$. By IH, $(\sigma\star\tau)(l\overline{x},A)=\sigma(l\overline{x},\tau(l\overline{x},A))$ and $(\sigma\star\tau)(r\overline{x},B)=\sigma(r\overline{x},\tau(r\overline{x},B))$. Thus we have that
    \begin{align*}
        (\sigma\star\tau)(\overline{x},A\to B)
         & =(\sigma\star\tau)(l\overline{x},A)\to(\sigma\star\tau)(r\overline{x},B) \\
         & = \sigma(l\overline{x},\tau(l\overline{x},A))\to \sigma(r\overline{x},\tau(r\overline{x},B)) \\
         & = \sigma(\overline{x},\tau(l\overline{x},A)\to\tau(r\overline{x},B)) \\
         & = \sigma(\overline{x},\tau(\overline{x},A\to B))
    \end{align*}
\end{proof}

\begin{lemma}\label{lemma:largest}
    $\lrcnclose{X}$ is closed under lericone substitutions and if $Y\subseteq X$ is closed under lericone substitutions, then $Y\subseteq \lrcnclose{X}$. Thus $\lrcnclose{X}$ is the largest subset of $X$ closed under lericone substitutions. 
\end{lemma}
\begin{proof}
    Much as in the proof of Theorem 3 of \citet{Leach-Krouse2024-LEALIT-4}. Explicitly, to show that $\lrcnclose{X}$ is closed under lericone substitutions, let $A\in\lrcnclose{X}$ and $\sigma$ be a lericone substitution. To see that $\sigma(\varepsilon,A)\in\lrcnclose{X}$, we need to show that for all lericone substitutions $\tau$, $\tau(\varepsilon,\sigma(\varepsilon,A))\in\lrcnclose{X}$. But by Lemma~\ref{lemma:star}, $\tau(\varepsilon,\sigma(\varepsilon,A))=\tau\star\sigma(\varepsilon,A)$. And since $\tau\star\sigma$ is a lericone substitution and $A\in\lrcnclose{X}$, it follows that $\tau\star\sigma(\varepsilon,A)\in X$ as required.
    
    Now note that if $Y\subseteq X$ is closed under lericone substitutions and $A\in Y$, then $A\in X$, so $Y\subseteq \lrcnclose{X}$. Thus $\lrcnclose{X}$ contains every subset of $X$ that is closed under lericone substitutions.
\end{proof}

Our first goal is to show that $\CLV=\lrcnclose{\logic{CL}}$. In preparation for that, we will prove some lemmas.

\begin{mdef}
    Let $\sigma$ be a lericone substitution and $f$ be an assignment. We define the assignment $(f\bullet\sigma)$ by $(f\bullet\sigma)(\overline{x},p)=f(\overline{x},\sigma(\overline{x},p))$. 
\end{mdef}

\begin{lemma}\label{Lemma:Bullet}
    For all $A\in\LL$, $(f\bullet\sigma)(\overline{x},A)=f(\overline{x},\sigma(\overline{x},A))$.
\end{lemma}
\begin{proof}
    By induction on $A$. The base case is immediate from the definition of $(f\bullet\sigma)$. We consider only the $\varepsilon$ part of the conditional case and leave the rest to the reader.

    For that case, we compute as follows:
    \begin{align*}
        (f\bullet\sigma)(\varepsilon,A\to B) & = \sup(1-(f\bullet\sigma)(c,A),(f\bullet\sigma)(c,B)) \\
            & = \sup(1-f(c,\sigma(c,A)),f(c,\sigma(c,B)) \\
            & = f(\varepsilon,\sigma(c,A)\to\sigma(c,B)) \\
            & = f(\varepsilon,\sigma(\varepsilon,A\to B))
    \end{align*}
\end{proof}

\begin{lemma}\label{Lemma:SkeletonsValid}
    $\CLV$ is closed under lericone substitutions: If $A\in\CLV$ and $\sigma(\varepsilon,A)=B$, then $B\in\CLV$.
\end{lemma}
\begin{proof}
    We prove the contrapositive. So, suppose $\sigma(\varepsilon,A)=B$ but $B\not\in\CLV$.
    Then there is an assignment $f$ such that $f(\varepsilon, B)=0$. By Lemma~\ref{Lemma:Bullet}, $(f\bullet \sigma)(\varepsilon,A)=f(\varepsilon,\sigma(\varepsilon,A))=f(\varepsilon,B)=0$. Thus $A\not\in\CLV$.
\end{proof}

Lemmas \ref{lemma:largest} and \ref{Lemma:SkeletonsValid} establish that $\CLV\subseteq\lrcnclose{\logic{CL}}$. Let us turn to the converse. A few definitions and lemmas will speed us along.

\begin{mdef}[Skeletons]
    Let $A$ be a formula. $B\in\LL$ is a skeleton of $A\in\LL$ iff for some injective atomic lericone substitution $\iota$, $\iota(\varepsilon, A)=B$. $B$ is a skeleton iff there is some formula $A$ such that  $B$ is a skeleton of $A$.
\end{mdef}
Clearly every formula has many skeletons. 

\begin{mdef}
    Let $\iota$ be an atomic injective lericone substitution. We define the plain substitution $\iota^{-1}$ as follows:
    \begin{displaymath}
        \iota^{-1}(p) =\left\{
            \begin{array}{rl}
                q   & \text{if for some $\overline{y}$, }\iota(\overline{y},q)=p \\
                p   & \text{otherwise}
            \end{array}\right.
    \end{displaymath}
\end{mdef}

\begin{lemma}
    $\iota^{-1}$ is well-defined and for all formulas $A$ and all $\overline{x}\in\lrcn$, $\iota^{-1}(\iota(\overline{x},A))=A$.
\end{lemma}
\begin{proof}
    Suppose that $\iota(\overline{y},q)=\iota(\overline{z},r)$. Then since $\iota$ is injective, $q=r$. Thus $\iota^{-1}$ is well-defined.

    To see that $\iota^{-1}(\iota(\overline{x},A))=A$, a straightforward induction on $A$ suffices.
\end{proof}

\begin{lemma}\label{lemma:any_skeleton}
    Let $X$ be a set of formulas. Then, for every $A\in\LL$, and every skeleton $B$ of $A$, $A\in\lrcnclose{X}$ iff $B\in\lrcnclose{X}$.
\end{lemma}
\begin{proof}
    The `only if' direction is immediate. For the other direction, suppose $B\in \lrcnclose{X}$, for each skeleton $B$ of $A$. Pick one such $B$. By definition, there is an atomic injective $\iota$ such that $\iota(\varepsilon,A)=B$. 

    Since $\lrcnclose{X}$ is closed under lericone substitutions, it's closed under plain substitutions as well. Thus $\iota^{-1}(\iota(\varepsilon,A))=A\in\lrcnclose{X}$ as required.
\end{proof}

\begin{lemma}\label{Lemma:SubstitutionsContainedInValid}
$\lrcnclose{\logic{CL}}$ is contained in $\CLV$. 
\end{lemma}
\begin{proof}
Suppose $B\not\in\CLV$. By Lemma~\ref{lemma:any_skeleton}, we may without loss of generality assume that $B$ is a skeleton.
Since $B\not\in\CLV$, there is an assignment $f$ such that $f(\varepsilon,B)=0$.
Since $B$ is a skeleton, we can define a \logic{CL}-assignment $g$ as follows. 
\[
g(p)=\left\{
        \begin{array}{rl}
            f(\overline{x},p) & \text{ if } p \text{ occurs in $B$ and $\overline{x}=\lrcn(B[p])$} \\
            1 & \text{ otherwise}.
        \end{array}\right.
\]
By an induction on $B$ we can show that $f(\varepsilon, B)=g(B)=0$.
By definition, $\lrcnclose{\logic{CL}}\subseteq \logic{CL}$, so $g(B)=0$ implies $B\not\in\lrcnclose{\logic{CL}}$.

\end{proof}

Thus, we've now established\
\begin{theorem}
    $\lrcnclose{\logic{CL}}=\CLV$
\end{theorem}

    \subsection{$\CLV$ is A Relevant Logic}\label{sec:CLV_is_relevant}
    
    $\CLV$ is stronger than the logic \logic{BM}. 
The formula  $(p\to q)\lor(q\to r)$ is contained in the former but not in the latter. While this formula is not  attractive from a relevant-logical point of view, it is surprising that some of the paradoxes of implication survive even the stringent standards of lericone-sensitive assignments. 

    On the other hand, given that it does in fact contain some of the paradoxes of material implication, and seems to do so in virtue of its being quite close in spirit to classical logic, it comes as quite a surprise to find (as we are about to show) that $\CLV$ nonetheless enjoys the variable sharing property. To prove this, we will need a  definition and a lemma. First, the definition of polarity of an $\LRN$-sequence.

\begin{mdef}[Polarity]
    The \emph{polarity} of $\overline{x}\in\LRN$ is defined as follows. 
    \begin{itemize}
        \item The polarity of $\varepsilon$ is positive. 
        \item If the polarity of $\overline{y}$ is positive(negative), then the polarity of $n\overline{y}$ is negative(positive). 
        \item If the polarity of $\overline{y}$ is positive(negative), then the polarity of $l\overline{y}$ is negative(positive). 
        \item If the polarity of $\overline{y}$ is positive(negative), 
then the polarity of $r\overline{y}$ is positive(negative). 
    \end{itemize}
\end{mdef}

This definition of polarity matches the usual definition, when one is not considering $c$.\footnote{Logan~\citeyearpar{logan2022} provides an example of the usual definition, in a non-lericone context.}
We will only be concerned with polarity in contexts where $c$ will not arise.

\begin{mdef}
    Suppose $A$ and $C$ share no atoms. We define the following functions:
    \begin{align*}
    f^+_{A}(\overline{y},p) & =\left\{
        \begin{array}{rl} 
            1 & \text{if $\overline{y}=\overline{x}c$, $\overline{x}$ is positive, and $\lrcn((A\to C)[p])=\overline{x}c$} \\
            0 & \text{otherwise}
        \end{array}\right. \\
        f^-_{A}(\overline{y},p)& =\left\{
        \begin{array}{rl} 
            0 & \text{if $\overline{y}=\overline{x}c$, $\overline{x}$ is positive, and $\lrcn((C\to A)[p])=\overline{x}c$} \\
            1 & \text{otherwise}
        \end{array}\right. \\
    \end{align*}
\end{mdef}

Note that aside from the initial caveat, the choice of $C$ in the above definition is unimportant, and that it works just as well to define $f^-_{A}=1-f^+_{A}$. But the intuitive connection between the lemma here and the theorem below is better if we put the definition this way. Finally, note that since $A$ and $C$ don't share any atoms, no subformula of $A$ is a subformula of $C$.

\begin{lemma}
    If $B$ is a subformula of $A$, then $f^+_A(\overline{x}c,B)=1$ if $\overline{x}$ is positive and $\lrcn((A\to C)[B])=\overline{x}c$ and $f^+_A(\overline{x}c,B)=0$ if $\overline{x}$ is negative and $\lrcn((A\to C)[B])=\overline{x}c$.
\end{lemma}
\begin{proof}
    By induction on $B$. If $B=p$ is an atom, the result is immediate from the definition of $f^+_A$. 

    For conjunctions note that $f^+_A(\overline{x}c,B_1\land B_2)=\inf(f^+_A(\overline{x}c,B_1),f^+_A(\overline{x}c,B_2))$ and if $\lrcn((A\to C)[B_1\land B_2])=\overline{x}c$, then $\lrcn((A\to C)[B_1])=\overline{x}c$ and $\lrcn((A\to C)[B_2])=\overline{x}c$. So if $\overline{x}$ is positive, then by IH $f^+_A(\overline{x}c,B_i)=1$ and thus $f^+_A(\overline{x}c,B_1\land B_2)=1$. On the other hand, if $\overline{x}$ is negative, then by IH $f^+_A(\overline{x}c,B_i)=0$ and thus $f^+_A(\overline{x}c,B_1\land B_2)=0$. Mutatis mutandis the same arguments work for disjunctions.

    For negations, note that $f^+_A(\overline{x}c,\neg B')=1-f^+_A(n\overline{x}c,B')$. If $\lrcn((A\to C)[\neg B'])=\overline{x}c$, then $\lrcn((A\to C)[B'])=n\overline{x}c$. So if $\overline{x}$ is positive, then by IH $f^+_A(n\overline{x}c,B')=0$ and thus $f^+_A(\overline{x}c,\neg B')=1$. On the other hand, if $\overline{x}$ is negative, then by IH $f^+_A(n\overline{x}c,B')=1$ and thus $f^+_A(\overline{x}c,\neg B')=0$. 

    Finally, for conditionals note that $f^+_A(\overline{x}c,B_1\to B_2)=\sup(1-f^+_A(l\overline{x}c,B_1),f^+_A(r\overline{x}c,B_2))$. If $\lrcn((A\to C)[B_1\to B_2])=\overline{x}c$, then $\lrcn((A\to C)[B_1])=l\overline{x}c$ and $\lrcn((A\to C)[B_2])=r\overline{x}c$. So if $\overline{x}$ is positive, then by IH $f^+_A(l\overline{x}c,B_1)=0$ and thus $f^+_A(\overline{x}c,B_1\to B_2)=1$. On the other hand, if $\overline{x}$ is negative, then by IH $f^+_A(l\overline{x}c,B_1)=1$ and $f^+_A(r\overline{x}c,B_2)=0$ and thus $f^+_A(\overline{x}c,B_1\to B_2)=0$. 
\end{proof}

\begin{lemma}
    If $B$ is a subformula of $A$, then $f^-_A(\overline{x}c,B)=0$ if $\overline{x}$ is positive and $\lrcn((A\to C)[B])=\overline{x}c$ and $f^-_A(\overline{x}c,B)=1$ if $\overline{x}$ is negative and $\lrcn((A\to C)[B])=\overline{x}c$.
\end{lemma}
\begin{proof}
    By induction on $B$. If $B=p$ is an atom, the result is immediate from the definition of $f^-_A$. 

    For conjunctions note that $f^-_A(\overline{x}c,B_1\land B_2)=\inf(f^-_A(\overline{x}c,B_1),f^-_A(\overline{x}c,B_2))$ and if $\lrcn((A\to C)[B_1\land B_2])=\overline{x}c$, then $\lrcn((A\to C)[B_1])=\overline{x}c$ and $\lrcn((A\to C)[B_2])=\overline{x}c$. So if $\overline{x}$ is positive, then by IH $f^-_A(\overline{x}c,B_i)=0$ and thus $f^-_A(\overline{x}c,B_1\land B_2)=1$. On the other hand, if $\overline{x}$ is negative, then by IH $f^-_A(\overline{x}c,B_i)=1$ and thus $f^-_A(\overline{x}c,B_1\land B_2)=1$. Mutatis mutandis the same arguments work for disjunctions.

    For negations, note that $f^-_A(\overline{x}c,\neg B')=1-f^-_A(n\overline{x}c,B')$. If $\lrcn((A\to C)[\neg B'])=\overline{x}c$, then $\lrcn((A\to C)[B'])=n\overline{x}c$. So if $\overline{x}$ is positive, then by IH $f^-_A(n\overline{x}c,B')=0$ and thus $f^-_A(\overline{x}c,\neg B')=1$. On the other hand, if $\overline{x}$ is negative, then by IH $f^-_A(n\overline{x}c,B')=1$ and thus $f^-_A(\overline{x}c,\neg B')=0$. 

    Finally, for conditionals note that $f^-_A(\overline{x}c,B_1\to B_2)=\sup(1-f^-_A(l\overline{x}c,B_1),f^-_A(r\overline{x}c,B_2))$. If $\lrcn((A\to C)[B_1\to B_2])=\overline{x}c$, then $\lrcn((A\to C)[B_1])=l\overline{x}c$ and $\lrcn((A\to C)[B_2])=r\overline{x}c$. So if $\overline{x}$ is positive, then by IH $f^-_A(l\overline{x}c,B_1)=0$ and thus $f^-_A(\overline{x}c,B_1\to B_2)=1$. On the other hand, if $\overline{x}$ is negative, then by IH $f^-_A(l\overline{x}c,B_1)=1$ and $f^-_A(r\overline{x}c,B_2)=0$ and thus $f^-_A(\overline{x}c,B_1\to B_2)=0$. 
\end{proof}

\begin{lemma}\label{Lemma:h-definition}
    Suppose that $A$ and $B$ share no atoms. Define the assignment $h$ by
    \[
        h(\overline{x},p)=\left\{
        \begin{array}{rl}
            f^+_A(\overline{x}, p) & \text{ if } p \text{ occurs in $A$ and $\overline{x}=\lrcn((A\to B)[p])$} \\
            f^-_B(\overline{x}, p) & \text{ if } p \text{ occurs in $B$ and $\overline{x}=\lrcn((A\to B)[p])$} \\
            1 & \text{ otherwise}.
        \end{array}\right.
    \]
    Then if $C$ is a subformula of $A$ then $h(\lrcn((A\to B)[C]),C)=f^+_A(\lrcn((A\to B)[C]),C)$ and if $C$ is a subformula of $B$ then $h(\lrcn((A\to B)[C]),C)=f^-_B(\lrcn((A\to B)[C]),C)$
\end{lemma}
\begin{proof}
    By induction on $C$, separately for each conclusion. For atoms the result is immediate from the definition of $h$. We sample a selection of the remaining clauses and leave the rest to the reader.

    Suppose $C=D_1\land D_2$ is a subformula of $A$, and say $\lrcn((A\to B)[C])=\overline{x}c$. As $\lrcn((A\to B)[C])=\lrcn((A\to B)[D_i])$,
    it follows that
    \[
    h(\overline{x}c,C)=\min(h(\overline{x}c,D_1), h(\overline{x}c,D_2)).
    \]
    
     By IH, 
    \[
    \min(h(\overline{x}c,D_1), h(\overline{x}c,D_2))=\min(f^+_A(\overline{x}c,D_1),f^+_A(\overline{x}c,D_2))=f^+_A(\overline{x}c,C),
    \]
    which was to be proved.

    Suppose $C=D_1\lor D_2$ is a subformula of $B$, and say $\lrcn((A\to B)[C])=\overline{x}c$. Then, as $\lrcn((A\to B)[C])=\lrcn((A\to B)[D_i])$,
    \[
    h(\overline{x}c,C)=\max(h(\overline{x}c,D_1), h(\overline{x}c,D_2)). 
    \]
    
    By IH,
    \[
    \max(h(\overline{x}c,D_1), h(\overline{x}c,D_2))=\max(f^-_B(\overline{x}c,D_1),f^-_B(\overline{x}c,D_2))=f^-_B(\overline{x}c,C),
    \]
    which completes the case.

    Suppose $C=\neg D$ is a subformula of $A$, and say $\lrcn((A\to B)[C])=\overline{x}c$. As $n\overline{x}c=\lrcn((A\to B)[D])$, it follows that
    \begin{center}
    \begin{tabular}{ccc}
    $h(\overline{x}c,C)$ & = &  $1-h(n\overline{x}c,D)$\\ 
    & =& $1-f^+_A(n\overline{x}c,D)$ \\ 
    & =& $f^+_A(\overline{x}c,C)$ \\ 
    \end{tabular}
    \end{center}
    The transition from the first to the second line is justified by IH, and the remainder are justified by the definition of assignments. 
    
    Suppose $C=D_1\to D_2$ is a subformula of $B$, and say $\lrcn((A\to B)[C])=\overline{x}c$. 
    As $l\overline{x}c=\lrcn((A\to B)[D_1])$ and $r\overline{x}c=\lrcn((A\to B)[D_2])$, it follows that 
    \[
    h(\overline{x}c,C)=\max(1-h(l\overline{x}c,D_1), h(r\overline{x}c,D_2)).
    \]
    Therefore, 
    \[
    \max(1-h(l\overline{x}c,D_1), h(r\overline{x}c,D_2))=\max(1-h(\lrcn((A\to B)[D_1]),D_1), h(\lrcn((A\to B)[D_2]),D_2)).
    \]
    By IH, this implies that the right-hand side is identical to
    \[
    \max(1-f^-_B(\lrcn((A\to B)[D_1]),D_1), f^-_B(\lrcn((A\to B)[D_2]),D_2)), 
    \]
    
    which in turn is identical to $f^-_B(\lrcn((A\to B)[C]),C)$,
    as desired.

\end{proof}

With this lemma in hand, we can prove that $\CLV$ enjoys variable sharing. 
\begin{theorem}\label{LRCN_Vs}
    $\CLV$ enjoys the variable sharing property. 
\end{theorem}
\begin{proof}
    Suppose that $A$ and $B$ share no atoms and define $h$ as above. By the preceding lemmas, $h(c,A)=1$ and $h(c,B)=0$. It follows that $h(\varepsilon, A\to B)=0$, and thus $A\to B\not\in\CLV$. Contraposing, if $A\to B\in\CLV$, then $A$ and $B$ share an atom.
\end{proof}
Inspection of the proof reveals that we can strengthen the result to the following new form of variable sharing. 
\begin{cor}
If $A\to B\in \CLV$, then $A$ and $B$ share an atom with the same lericone sequence in $A\to B$.    
\end{cor}
\begin{proof}
Suppose that $A$ and $B$ do not share an atom with the same lericone sequence. Let $A'\to B'$ be a skeleton that results from applying an atomic injective substitution to $A\to B$. It follows that $A'$ and $B'$ do not share any atoms. By the previous theorem, $A'\to B'\not\in\CLV$. As $\CLV$ is closed under lericone substitutions, it follows that $A\to B\not\in \CLV$. Therefore, by contraposing, if $A\to B\in \CLV$, $A$ and $B$ share an atom with the same lericone sequence in $A\to B$.    

\end{proof}
Since $\CLV$ enjoys variable sharing, it is a relevant logic. 
We note that, since  $\CLV$ contains $(p\to q)\lor(q\to r)$, it is incomparable with the best known relevant logic, Anderson and Belnap's logic \logic{R}, which contains violations of lericone substitution closure, such as $(p\to(p\to q))\to(p\to q)$.
In fact, it is incomparable with the logic \logic{RMingle}, obtained from \logic{R} by the addition of $A\to(A\to A)$, due to the same disjunction.
As a consequence, $\CLV$ is also incomparable with $\logic{TMingle}$, which enjoys the variable sharing property, unlike \logic{RMingle}.\footnote{The invalidity of the displayed disjunction in \logic{RMingle} can be shown using John Slaney's program $\tt{MaGIC}$, which is available at {\tt {http://users.cecs.anu.edu.au/~jks/magic.html}}. We leave this as an exercise for the interested reader.} 
Closure under lericone substitution identifies a different subfamily of the relevant logics from the usual suspects. 

A noteworthy observation to make at this point is that in light of Theorem~\ref{LRCN_Vs}, the proofs of Corollary~\ref{BM_var_shar} and Corollary~\ref{B_var_shar} need no longer be parasitic on variable sharing results for $\mathbf{R}$---we need only observe that since the logics in question are lericone-invariant sublogics of classical logic, they are ipso facto sublogics of $\CLV$. This fact is sufficient to carry, in the proofs of the above corollaries, the weight formerly carried by known variable sharing results for $\mathbf{R}$.

    \subsection{$\vDash_{\LRCN}$ is compact}

We end the `nice features of $\CLV$ and related systems' theme by showing that $\vDash_{\LRCN}$ is compact. 
To begin, we establish some notation. 
For a lericone substitution $\sigma$, say that $\sigma(\overline{x},X)=\set{\sigma(\overline{x},B):B\in X}$. 
Where $\sigma$ is a lericone substitution, for a set of formulas $X$, let $X^\sigma=\sigma(\varepsilon,X)$, and for a formula $A$, $A^\sigma=\sigma(\varepsilon,A)$. 
\begin{lemma}
Suppose $\iota$ is an injective atomic lericone substitution, that $X=Y^\iota$, for some set of formulas $Y$, and that $A=B^\iota$, for some formula $B$.
If $X\vDash_{\logic{CL}}A$, then $X\vDash_{\LRCN} A$. 
\end{lemma}
\begin{proof}
Suppose $X\not\vDash_{\LRCN} A$. So, there is a $\LRCN$-assignment $f$ such that $f(\varepsilon, B)=1$, for each $B\in X$, and $f(\varepsilon,A)=0$. 
We want to construct a \logic{CL}-assignment $g$ ``matching'' $f$. 
First, note that for every atom $p$ in $X\cup\set{A}$, for all formulas $B_1,B_2\in X\cup\set{A}$, $\lrcn(B_1[q_1])=\lrcn(B_2[q_2])$, where $q_1$ and $q_2$ are occurrences of $p$. 
Suppose that $\lrcn(B_1[q_1])\neq\lrcn(B_2[q_2])$, and say that $\overline{x}=\lrcn(B_1[q_1])$ and $\overline{y}=\lrcn(B_2[q_2])$. 
Then for some $r,s\in\At$, $\iota(\overline{x},r)=p$ and $\iota(\overline{y},s)=p$, which contradicts the injectivity of $\iota$. 
Therefore, each occurrence of $p$ in $X\cup\set{A}$ has the same $\lrcn$-sequence, which we will denote by $\lrcn(p)$. 

Define the $\logic{CL}$-assignment $g$ by defining 
\[
g(p)=\left\{
        \begin{array}{rl}
            f(\lrcn(p), p) & \text{ if } p \text{ occurs in $X\cup\set{A}$}  \\
            1 & \text{ otherwise}.
        \end{array}\right.
\]
The function $g$ is well-defined because for each $p\in \At$, $\lrcn(p)$ is well-defined.
We claim that for $B\in X\cup\set{A}$, $g(B)=f(\varepsilon, B)$. 
By an induction on structure, for each $B\in X\cup\set{A}$, for each subformula $C$ of $B$, $g(C)=f(\lrcn(B[C]), C)$. 

It follows that $g(B)=1$, for all $B\in X$, and $g(A)=0$. Therefore, $X\not\vDash_{\logic{CL}}A$, as desired.
\end{proof}
Next, we note a lemma whose proof is straightforward. 
\begin{lemma}
    Suppose that $\iota$ is an injective atomic lericone substitution, that $X=Y^\iota$, for some $Y\subseteq \LL$, and that $A=B^\iota$, for some $B\in\LL$.
If $X\vDash_{\LRCN}A$, then $X\vDash_{\logic{CL}} A$. 
\end{lemma}
\begin{proof}
The proof is left to the reader.
\end{proof}
Putting the preceding lemmas together we have the following corollary. 
\begin{cor}\label{Cor:LRCNandCLequiv}
Suppose $\iota$ is an injective atomic lericone substitution, that $X=Y^\iota$, for some set of formulas $Y$, and that $A=B^\iota$, for some formula $B$.
Then,  $X\vDash_{\LRCN}A$ iff $X\vDash_{\logic{CL}} A$. 
\end{cor}
We will note another straightforward lemma. 
\begin{lemma}\label{Lemma:LRCNCOnClosure}
The following are equivalent, for all $X\subseteq\LL$ and all  $A\in\LL$. 
\begin{itemize}
\item $X\vDash_{\LRCN} A$.
\item $X^\iota\vDash_{\LRCN} A^\iota$, for all atomic injective $\iota$. 
\end{itemize}
\end{lemma}
\begin{proof}
The proof is left to the reader.
\end{proof}
Putting these pieces together, we get compactness for $\vDash_{\LRCN}$. 
\begin{theorem}
For $X\subseteq\LL$ and $A\in\LL$, if $X\vDash_{\LRCN}A$, then for some finite $Y\subseteq X$, $Y\vDash_{\LRCN}A$.
\end{theorem}
\begin{proof}
Suppose $X\vDash_{\LRCN}A$. By Lemma \ref{Lemma:LRCNCOnClosure}, we can assume that $X$ and $A$ are the images of some atomic injective $\iota$. 
It then follows by Corollary \ref{Cor:LRCNandCLequiv} that $X\vDash_{\logic{CL}}A$. By the compactness of classical logic, 
for some finite $Y\subseteq X$, $Y\vDash_{\logic{CL}}A$. By Corollary \ref{Cor:LRCNandCLequiv}, $Y\vDash_{\LRCN}A$, as desired. 
\end{proof}

\subsection{Faithful $\CLV$}

Recall the distinction between \emph{faithful} lericone substitutions and lericone substiutions \emph{simpliciter}; this distinction, we showed, is reflected in the relationship between $\mathbf{B}$ (closed under faithful substitutions) and $\mathbf{BM}$ (closed under all lericone substitutions). By encoding Involutive Transparency in its axioms, $\mathbf{B}$ acts as a sort of faithful counterpart to $\mathbf{BM}$.

We can now turn our attention to a supersystem of $\CLV$ that bears to $\CLV$ the same relationship that $\mathbf{B}$ bears to $\mathbf{BM}$.

\begin{mdef}
    $X^{\mathsf{LRCN}^{\sim}}=\lbrace A:\sigma(\varepsilon,A)\in X\mbox{ for all faithful lericone substitutions }\sigma\rbrace$
\end{mdef}

\begin{mdef}
    A set of sentences $X$ is closed under faithful lericone substitutions when for all faithful lericone substitutions $\sigma$, if $A\in X$ then $\sigma(\varepsilon,A)\in X$.
\end{mdef}

Recall the definition of the operation $\star$. We make a few observations. Note first that Lemma \ref{lemma:star} holds \emph{a fortiori} in case $\sigma$ and $\tau$ are faithful lericone substitutions. This observation allows us to infer a second lemma:

\begin{lemma}
     If $\sigma,\tau$ are faithful lericone substitutions, then $\sigma\star \tau$ is a faithful lericone substitution.
\end{lemma}

\begin{proof}
    Suppose $\sigma,\tau$ to be faithful and suppose that $\overline{x}\sim\overline{y}$. Then by Lemma \ref{lemma:star},   $(\sigma\star\tau)(\overline{x},A)$ is $\sigma(\overline{x},\tau(\overline{x},A))$. By assumption that $\sigma$ and $\tau$ are faithful, $\sigma(\overline{x},\tau(\overline{x},A))=\sigma(\overline{x},\tau(\overline{y},A))=\sigma(\overline{y},\tau(\overline{y},A)$, which, by Lemma \ref{lemma:star} once more, is equal to $(\sigma\star\tau)(\overline{y},A)$. As $A$, $\overline{x}$, and $\overline{y}$ were arbitrary, we conclude that $\sigma\star\tau$ is faithful.
\end{proof}

Knowing that $\sigma\star\tau$ is faithful in case $\sigma$ and $\tau$ are allows us to modify the logic of the proof of Lemma 5 to infer the following:

\begin{lemma}
    For a set of formulas $X$, $X^{\LRCN^{\sim}}$ is the largest subset of $X$ closed under faithful lericone substitutions.
\end{lemma}

\begin{mdef}
    A faithful lericone-sensitive truth-value assignment is a lericone-sensitive truth-value assignment such that for all $\overline{x}\overline{y}\in\mathsf{LRCN}$ and formulas $p\in\At$, $f(\overline{x}\overline{y},p)=f(\overline{x}nn\overline{y},p)$. 
\end{mdef}

\begin{lemma}
    If $f$ is a faithful lericone-sensitive truth-value assignment, then for all $A\in\LL$, $f(\overline{xy},A)=f(\overline{x}nn\overline{y},A)$
\end{lemma}
\begin{proof}
    By induction on $A$.
\end{proof}

As Lemma 6 holds \emph{a fortiori} in case $f$ is a faithful lericone-sensitive assignment and $\sigma$ is a faithful lericone substitution, we can infer the following lemma almost for free:

\begin{lemma}
    $\mathbf{CLV}^{\sim}$ is closed under faithful lericone substitutions, \emph{i.e.}, $(\mathbf{CLV}^{\sim})^{\LRCN^{\sim}}$ $=$ $\mathbf{CLV}^{\sim}$.
\end{lemma}

This allows us to infer that $\mathbf{CLV}^{\sim}\subseteq\mathbf{CL}^{\LRCN^{\sim}}$. To show the converse, we note again that many lemmas established in pursuit of showing the relationship between $\mathbf{CLV}$ and $\mathbf{CL}^{\LRCN}$ carry over immediately to the case of \emph{faithful} lericone substitutions.

\begin{mdef}
Let $B\in\mathcal{L}$ be a skeleton of a formula $A$. We call $B$ a \emph{faithful skeleton} if there exists an injective atomic faithful lericone substitution $\sigma$ such that $\sigma(\varepsilon, A)=B$.
\end{mdef}

\begin{lemma}
    For $X$ a set of formulas, for all $A\in\mathcal{L}$ and every faithful skeleton $B$ of $A$, $A\in X^{\LRCN^{\sim}}$ iff $B\in X^{\LRCN^{\sim}}$.
\end{lemma}

\begin{proof}
    Just as every plain substitution is a lericone substitution, so, too, is every plain substitution a faithful lericone substitution. Consequently, selecting an atomic injective faithful $\sigma$ such that $\sigma(\varepsilon,A)=B$ allows us to rehearse the reasoning of the proof of Lemma \ref{lemma:any_skeleton} to the case of $X^{\LRCN^{\sim}}$.
\end{proof}

The reasoning of the proof of Lemma \ref{Lemma:SubstitutionsContainedInValid}---that from any lericone counterexample $f$, one can construct a classical counterexample $g$---applies \emph{a fortiori} to faithful lericone assignments $f$, allowing us to infer that $\mathbf{CL}^{\LRCN^{\sim}}\subseteq\mathbf{CLV}^{\sim}$, whence we may conclude the following:

\begin{theorem}
    $\mathbf{CL}^{\LRCN^{\sim}}=\mathbf{CLV}^{\sim}$
\end{theorem}

We can note the distinctness of $\mathbf{CLV}$ and $\mathbf{CLV}^{\sim}$ by observing that $A\rightarrow\neg\neg A$ is not valid in $\mathbf{CLV}$ although it is valid in $\mathbf{CLV}^{\sim}$. Consequently, $\mathbf{CLV}\subsetneq\mathbf{CLV}^{\sim}$. For this reason, that $\mathbf{CLV}$ is a relevant logic need not immediately entail that $\mathbf{CLV}^{\sim}$ is a relevant logic. We start the task of showing it to be relevant via the following lemma:

\begin{lemma}
    $\bar{x}nn\bar{y}$ is positive (negative) iff $\bar{x}\bar{y}$ is positive (negative)
\end{lemma}

This means that with respect to the subformulas that matter, the assignment $h$ defined in Lemma \ref{Lemma:h-definition} is a faithful assignment. But it need not be faithful in general, requiring that we show how to produce a faithful version of $h$. 

\begin{lemma}\label{Lemma:h-sim-definition}
   Define the assignment $h^{\sim}$ by
    \[
        h^{\sim}(\overline{x},p)=\left\{
        \begin{array}{rl}
            f^+_A(\overline{x}, p) & \text{ if } p \text{ occurs in $A$ and $\overline{x}\sim\lrcn((A\to B)[p])$} \\
            f^-_B(\overline{x}, p) & \text{ if } p \text{ occurs in $B$ and $\overline{x}\sim\lrcn((A\to B)[p])$} \\
            1 & \text{ otherwise}.
        \end{array}\right.
    \]

   \noindent Then $h^{\sim}$ both has the properties described in Lemma \ref{Lemma:h-definition} and is faithful.
\end{lemma}

\begin{proof}
    The same induction as in Lemma \ref{Lemma:h-definition} carries over to establish that the same holds for $h^{\sim}$. To show that $h^{\sim}$ is faithful, \emph{n.b.}, that if $p$ occurs in $A$ (respectively, $B$) and $\bar{x}\sim\bar{y}$, then $f^{+}_{A}(\bar{x},p)=f^{+}_{A}(\bar{y},p)$ (respectively, $f^{-}_{B}(\bar{x},p)=f^{-}_{B}(\bar{y},p)$). Otherwise, $h^{\sim}(\bar{x},p)=1=h^{\sim}(\bar{y},p)$ independently of choice of $\bar{x}$.
\end{proof}

\noindent Following the argument of Theorem \ref{LRCN_Vs} but noting that $h^{\sim}$ is faithful allows us to infer that:

\begin{theorem}
    $\mathbf{CLV}^{\sim}$ enjoys the variable sharing property.
\end{theorem}

\noindent This, moreover, allows us to infer that $\mathbf{CLV}^{\sim}$ has a slightly weaker variable sharing property of \emph{faithful} lericone relevance:

\begin{cor}
    If $A\rightarrow B\in\mathbf{CLV}^{\sim}$ then $A$ and $B$ share an atom $p$ with equivalent $\LRN$ sequences in both $A$ and $B$.
\end{cor}

\section{Tableaux for Classical Lericone Consequence}\label{Section:Tableaux}

We take the preceding to demonstrate that $\CLV$ is an interesting system that deserves the attention of logicians, particularly relevant logicians. But as it stands, we have not provided anything resembling proof theory for $\CLV$. We conclude the paper by providing a tableau system for $\vDash_{\LRCN}$ and providing the usual metatheory for it. 

\subsection{Tableaux for $\CLV$}

Where $X\cup\{A\}$ is a set of formulas we call an expression of the form $X\Yright A$ a \textit{sequent}. Given a sequent $X\Yright A$, its initial tableau has a single branch containing, for each $B\in X$, a triple of the form $\langle \varepsilon,1,B\rangle$ and also containing the triple $\langle \varepsilon,0,A\rangle$. A tableau can be extended according to the following rules:

\begin{multicols}{2}\begin{description}
    \item[Positive Conjunction Rule] 
    \begin{displaymath}
        \xymatrix@C=-10mm@R=3mm{\langle\overline{x},1,B_1\land B_2\rangle\ar[d] \\
        \txt{$\langle\overline{x},1,B_1\rangle$ \\ $\langle\overline{x},1,B_2\rangle$}}
    \end{displaymath}
    \item[Negative Conjunction Rule]
    \begin{displaymath}
        \xymatrix@C=-10mm@R=3mm{ & \langle\overline{x},0,B_1\land B_2\rangle\ar[dr]\ar[dl] & \\
        \langle\overline{x},0,B_1\rangle & & \langle\overline{x},0,B_2\rangle}
    \end{displaymath}
    \item[Positive Disjunction Rule]
    \begin{displaymath}
        \xymatrix@C=-10mm@R=3mm{ & \langle\overline{x},1,B_1\lor B_2\rangle\ar[dr]\ar[dl] & \\
        \langle\overline{x},1,B_1\rangle & & \langle\overline{x},1,B_2\rangle}
    \end{displaymath}
    \item[Negative Disjunction Rule]
    \begin{displaymath}
        \xymatrix@C=-10mm@R=3mm{\langle\overline{x},0,B_1\lor B_2\rangle\ar[d] \\
        \txt{$\langle\overline{x},1,B_1\rangle$ \\ $\langle\overline{x},1,B_2\rangle$} }
    \end{displaymath}
    \item[Positive Negation Rule] 
    \begin{displaymath}
        \xymatrix@C=-10mm@R=3mm{\langle\overline{x},1,\neg B\rangle\ar[d] \\
        \langle n\overline{x},0,B\rangle}
    \end{displaymath}
    \item[Negative Negation Rule] 
    \begin{displaymath}
        \xymatrix@C=-10mm@R=3mm{\langle\overline{x},0,\neg B\rangle\ar[d] \\
        \langle n\overline{x},1,B\rangle}
    \end{displaymath}
    \item[Positive Conditional Rule, $\varepsilon$ case]
    \begin{displaymath}
        \xymatrix@C=-10mm@R=3mm{& \langle\varepsilon,1,B_1\to B_2\rangle\ar[dr]\ar[dl] & \\
        \langle c,0,B_1\rangle & &\langle c,1,B_2\rangle}
    \end{displaymath}
    \item[Positive Conditional Rule, $\overline{x}\neq\varepsilon$ case]
    \begin{displaymath}
        \xymatrix@C=-10mm@R=3mm{& \langle\overline{x},1,B_1\to B_2\rangle\ar[dr]\ar[dl] & \\
        \langle l\overline{x},0,B_1\rangle & &\langle r\overline{x},1,B_2\rangle}
    \end{displaymath}
    \item[Negative Conditional Rule, $\varepsilon$ case]
    \begin{displaymath}
        \xymatrix@C=-10mm@R=3mm{\langle\varepsilon,0,B_1\to B_2\rangle\ar[d] \\
        \txt{$\langle c,1,B_1\rangle$ \\ $\langle c,0,B_2\rangle$}}
    \end{displaymath}
    \item[Negative Conditional Rule, $\overline{x}\neq\varepsilon$ case]
    \begin{displaymath}
        \xymatrix@C=-10mm@R=3mm{\langle\overline{x},0,B_1\to B_2\rangle\ar[d] \\
        \txt{$\langle l\overline{x},1,B_1\rangle$ \\ $\langle r\overline{x},0,B_2\rangle$}}
    \end{displaymath}
\end{description}\end{multicols}

A branch on a tableau \textit{closes} when there is a lericone sequence $\overline{x}$ and a formula $A$ so that both $\langle \overline{x},0,A\rangle$ and $\langle\overline{x},1,A\rangle$ occur somewhere in the tableau. Otherwise, the branch is \textit{open}. A tableau is closed when all of its branches are closed. We say that $X\vdash A$ when the initial tableau for $X\Yright A$ eventually closes.

We will now turn to showing that $X\vdash A$ iff $X\vDash_{\LRCN} A$. To that end, we first introduce some definitions.

\begin{mdef}
    A \emph{set of triples} is a set all of whose members have the form $\langle\overline{x},i,A\rangle$ with $\overline{x}\in\LRCN$, $i\in\{0,1\}$, and $A\in\LL$.
\end{mdef}

\begin{mdef}
    A lericone assignment $f$ \emph{conforms} to the set of triples $S$ just if $\langle\overline{x},i,A\rangle\in S$ only if $f(\overline{x},A)=i$.
\end{mdef}

\begin{mdef}
    Let $S$ be a set of triples. We define the $\langle\overline{x},i,A\rangle$-extensions of $S$ as follows: if $A\in\At$ or $\langle\overline{x},i,A\rangle\not\in S$, then $S$ is the only $\langle\overline{x},i,A\rangle$-extension of $S$. Otherwise,
    \begin{center}\begin{tabular}{|c|c|}\hline
        $\langle i,A\rangle$  &   $\langle\overline{x},i,A\rangle$-extensions of $S$ \\ \hline \hline
        $\langle 1,B_1\land B_2\rangle$  &  $S\cup\{\langle\overline{x},1,B_1\rangle,\langle\overline{x},1,B_2\rangle\}$ \\ 
        $\langle 0,B_1\land B_2\rangle$  &  $S\cup\{\langle\overline{x},0,B_1\rangle\}$ and $S\cup\{\langle\overline{x},0,B_2\rangle\}$ \\ \hline
        $\langle 1,B_1\lor B_2\rangle$  &  $S\cup\{\langle\overline{x},1,B_1\rangle\}$ and $S\cup\{\langle\overline{x},1,B_2\rangle\}$ \\ 
        $\langle 0,B_1\lor B_2\rangle$  &  $S\cup\{\langle\overline{x},0,B_1\rangle,\langle\overline{x},0,B_2\rangle\}$ \\ \hline
        $\langle 1,\neg B\rangle$  &  $S\cup\{\langle\overline{x},0,B\rangle\}$ \\
        $\langle 0,\neg B\rangle$  &  $S\cup\{\langle\overline{x},1,B\rangle\}$ \\ \hline
        & if $\overline{x}=\varepsilon$ \\
        $\langle 1,B_1\to B_2\rangle$  &  $S\cup\{\langle {c},0,B_1\rangle\}$ and $S\cup\{\langle {c},1,B_2\rangle\}$ \\ 
        $\langle 0,B_1\to B_2\rangle$  &  $S\cup\{\langle {c},1,B_1\rangle,\langle {c},0,B_2\rangle\}$ \\ \hline
        & if $\overline{x}\neq\varepsilon$ \\
        $\langle 1,B_1\to B_2\rangle$  &  $S\cup\{\langle l{\overline{x}},0,B_1\rangle\}$ and $S\cup\{\langle r{\overline{x}},1,B_2\rangle\}$ \\ 
        $\langle 0,B_1\to B_2\rangle$  &  $S\cup\{\langle l{\overline{x}},1,B_1\rangle,\langle r{\overline{x}},0,B_2\rangle\}$ \\ \hline
    \end{tabular}\end{center}
\end{mdef}

\begin{lemma}
    Let $S$ be a set of triples and $\langle\overline{x},i,A\rangle\in S$. Then $f$ conforms to $S$ iff $f$ conforms to at least one $\langle\overline{x},i,A\rangle$-extension of $S$.
\end{lemma}
\begin{proof}
    By examining cases. The atomic case is immediate. We examine two of the four conditional cases and leave all other cases to the reader.

    Suppose $\langle\overline{x},i,A\rangle=\langle\varepsilon,0,B_1\to B_2\rangle$. Since $f$ conforms to $S$, $f(\varepsilon,B_1\to B_2)=0$. Thus $\sup(1-f(c,B_1),f(c,B_2))=0$. It follows that $f(c,B_1)=1$ and $f(c,B_2)=0$. So $f$ conforms to $S\cup\{\langle c,1,B_1\rangle,\langle c,0,B_2\rangle\}$ as required.

    Suppose $\overline{x}\neq\varepsilon$ and $\langle\overline{x},i,A\rangle=\langle\overline{x},1,B_1\to B_2\rangle$. Since $f$ conforms to $S$, $f(\overline{x},B_1\to B_2)=1$. Thus $\sup(1-f(l\overline{x},B_1),f(r\overline{x},B_2))=1$. So either $f(l\overline{x},B_1)=0$ or $f(r\overline{x},B_2)=1$. Thus $f$ either conforms to $S\cup\{\langle l\overline{x},0,B_1\rangle\}$ or conforms to $S\cup\{\langle r\overline{x},1,B_1\rangle\}$, as required.
\end{proof}

\begin{theorem}[Soundness]\label{theorem:soundness-LRCN}
    If $X\vdash A$, then $X\vDash_{\LRCN} A$.
\end{theorem}
\begin{proof}
    Suppose $X\vdash A$. Let $f(\varepsilon,B)=1$ for all $B\in X$. If $f(\varepsilon,A)=0$, then $f$ conforms to the initial tableau for $X\Yright A$. So by the preceding lemma, in all extensions of the initial tableau for $X\Yright A$, $f$ conforms to the set of formulas on at least one branch. But since $X\vdash A$, all branches eventually contain both $\langle\overline{x},1,C\rangle$ and $\langle\overline{x},0,C\rangle$, which no assignment can conform to. So $f(\varepsilon,A)\neq 0$, as required.
\end{proof}

For completeness, we need two further definitions:
\begin{mdef}
    A set of triples is \emph{saturated} when $\langle\overline{x},i,A\rangle\in S$ only if $S$ contains one of its $\langle\overline{x},i,A\rangle$-extensions.
\end{mdef}
\begin{mdef}
     A tableau is \emph{complete} just if for each of its branches, the set of triples on that branch is saturated. 
\end{mdef}

Notice that if $X\not\vdash A$, then there will be an open branch on each completed tableau for $X\Yright A$.

\begin{theorem}[Completeness]
    If $X\not\vdash A$, then $X\not\vDash_{\LRCN} A$.
\end{theorem}
\begin{proof}
    Choose a completed tableau for $X\Yright A$ and an open branch $\mathcal{B}$ on it. Let $f_{\mathcal{B}}(\overline{x},p)=1$ iff $\langle\overline{x},1,p\rangle\in \mathcal{B}$. By induction on $C$, one sees that if $\langle\overline{x},i,C\rangle\in B$, then $f_{\mathcal{B}}(\overline{x},C)=i$. So since $\{\langle\varepsilon,1,D\rangle\mid D\in X\}\subseteq \mathcal{B}$ and $\langle\varepsilon,0,A\rangle\in \mathcal{B}$, $f\not\vDash X\Yright A$. Thus $X\not\vDash_{\LRCN} A$.
\end{proof}

\subsection{Tableaux for Faithful $\CLV$}

Let us return to the semantically-defined consequence relation of $\mathbf{CLV}^{\sim}$ to provide a tableau calculus for this system as well.

We can provide a tableau calculus for the faithful version of $\CLV$ by adding to the foregoing calculus the following rule:

\vspace{.2cm}

\begin{center}
    \begin{description}

\item[Faithfulness Rule] 
    \begin{displaymath}
        \xymatrix@C=-10mm@R=3mm{\langle \bar{x}nn\bar{y},i,B\rangle\ar[d] \\
        \langle \bar{x}\bar{y},i,B\rangle}
    \end{displaymath}

    \end{description}
\end{center}

\vspace{.2cm}

We say that $X\vdash^{\sim}A$ holds when the initial tableau for $X\Yright A$ closes for tableaux including the Faithfulness Rule.

We could also offer a an alternative method of defining the tableau calculus by stating that a branch in the calculus omitting the Faithfulness Rule \emph{faithfully closes} when there are lericone sequences $\bar{x}\sim\bar{y}$ such that both $\langle\bar{x},0,A\rangle$ and $\langle\bar{y},1,A\rangle$ appear on the branch and say that a tableau omitting the Faithfulness Rule faithfully closes if every branch faithfully closes. Then we could say that $X\vdash^{\sim}_{2}A$ holds when every tableau omitting the Faithfulness Rule faithfully closes.

\begin{mdef}
    Let $\bar{x}$ be a lericone sequence and define its \emph{faithful reduct} $\rho(\bar{x})$ recursively by saying that $\rho(\overline{x})=\overline{x}$ if $\overline{x}$ contains no occurrences of $nn$ and $\rho(\overline{x}nn\overline{y})=\rho(\overline{xy})$ otherwise.
\end{mdef}

\noindent The following is immediate:

\begin{lemma}\label{lemma:reduct-identity}
    $\bar{x}\sim\bar{y}$ iff $\rho(\bar{x})=\rho(\bar{y})$
\end{lemma}

\noindent In support of proving soundness, we will establish two lemmas.

\begin{lemma}\label{lemma:equivalence-faithful}
    $X\vdash^{\sim}A$ iff $X\vdash^{\sim}_{2}A$
\end{lemma}

\begin{proof}
    For left-to-right, consider a completed tableau for $X\Yright A$ without the Faithfulness Rule with a branch $\mathcal{B}$ that is not faithfully closed, \emph{i.e.}, for which there are no nodes $\langle\bar{x},1,B\rangle$ and $\langle{y},0,B\rangle$ on $\mathcal{B}$ for which $\bar{x}\sim\bar{y}$. By Lemma \ref{lemma:reduct-identity}, there are no such $\bar{x},\bar{y}$ where $\rho(\bar{x})=\rho(\bar{y})$. Consequently, no number of applications of the Faithfulness Rule will yield a $\bar{z}$ such that $\langle\bar{z},1,B\rangle$ and $\langle\bar{z},0,B\rangle$ are on an extension to that branch, \emph{i.e.}, $X\nvdash^{\sim}A$.

    For right-to-left, take a tableau for $X\Yright A$ without the Faithfulness Rule such that all branches faithfully close. Then there exist $\langle\bar{x},1,B\rangle$ and $\langle\bar{y},0,B\rangle$ on any branch $\mathcal{B}$ such that $\bar{x}\sim\bar{y}$. One can then apply the Faithfulness Rule to each node finitely many times to yield $\langle\rho(\bar{x}),1,B\rangle$ and $\langle\rho(\bar{y}),0,B\rangle$ on an extended branch $\mathcal{B}'$ without any novel branching. By Lemma \ref{lemma:reduct-identity}, $\rho(\bar{x})=\rho(\bar{y})$ whence the extended branch $\mathcal{B}'$ closes \emph{simpliciter}.
\end{proof}

\begin{lemma}\label{lemma:faithful-conforming}
    A faithful $f$ cannot conform to a set of triples including two triples $\langle\bar{x},1,C\rangle$ and $\langle\bar{y},0,C\rangle$ such that $\bar{x}\sim\bar{y}$.
\end{lemma}

\begin{proof}
    By definition of $f$'s conforming to $S$, if $f$ were to conform to both $\langle\bar{x},1,C\rangle$ and $\langle\bar{y},0,C\rangle$, then $f(\bar{x},C)=1$ and $f(\bar{y},C)=0$, contradicting the faithfulness of $f$.
\end{proof}

\noindent This allows us to prove soundness of the tableau system with respect to $\LRCN^{\sim}$:

\begin{theorem}[Soundness]
    If $X\vdash^{\sim} A$, then $X\vDash_{\LRCN^{\sim}} A$.
\end{theorem}

\begin{proof}
    Suppose that $X\vdash^{\sim}A$ and suppose for contradiction that $f$ is a faithful assignment such that $f(\epsilon,B)=1$ for each $B\in X$ and $f(\epsilon,A)=0$. By Lemma \ref{lemma:equivalence-faithful}, this means that for every tableau in the system without the Faithfulness Rule, we can follow the reasoning of Theorem \ref{theorem:soundness-LRCN} to ensure in every extension of the initial tableau for $X\Yright A$, $f$ must conform to the formulas on some branch. But Lemma \ref{lemma:equivalence-faithful} means that every branch includes some nodes $\langle \bar{x},1,C\rangle$ and $\langle\bar{y},0,C\rangle$ where $\bar{x}\sim\bar{y}$, \emph{i.e.}, $f(\bar{x},C)\neq f(\bar{y},C)$, contradicting the faithfulness of $f$ per Lemma \ref{lemma:faithful-conforming}.
\end{proof}

\noindent To prove completeness, we again navigate a handful of definitions and lemmas:

\begin{mdef}
    A lericone-sensitive assignment $f$ is \emph{atomically faithful} if for all atoms $p$ and sequences $\bar{x}\sim\bar{y}$, $f(\bar{x},p)=f(\bar{y},p)$.
\end{mdef}

\begin{lemma}\label{lemma:atomic-faithful}
    A lericone-sensitive assignment $f$ is atomically faithful if and only if it is faithful \emph{simpliciter}.
\end{lemma}

\begin{proof}
    Right-to-left is immediate. Left-to-right follows from an induction on complexity of formulas, \emph{e.g.}, supposing that $\bar{x}\sim\bar{y}$ then:

    \begin{center}
        $f(\bar{x},B\wedge C)=\inf(f(\bar{x},B),f(\bar{x},C))=\inf(f(\bar{y},B),f(\bar{y},C))=f(\bar{y},B\wedge C)$
    \end{center}

\end{proof}

\begin{lemma}\label{lemma:find-a-faithful-counterexample}
    Let $X$ be a set of formulas with signed atoms $\mathsf{At}(X)=\lbrace \langle \bar{x},p\rangle\mid \bar{x}=\lrcn(B[p])\mbox{ for some }B\in X\rangle$ with respect to which it is atomically faithful. Then there exists a faithful $f'$ such that $f'(\varepsilon,B)=f(\varepsilon,B)$ for all $B\in X$.
\end{lemma}

\begin{proof}
    Define $f'$ so that

    \[
f'(p)=\left\{
        \begin{array}{rl}
            f(\overline{x},p) & \text{ if } p \text{ occurs in $B$ and $\overline{x}\sim\lrcn(B[p])$} \\
            1 & \text{ otherwise}.
        \end{array}\right.
\]

Then if $f$ is atomically faithful, so is $f'$ and moreover, by Lemma \ref{lemma:atomic-faithful}, $f'$ is faithful without qualification.
\end{proof}

\begin{mdef}
    Let $S$ be a set of triples. Its faithful closure $S^{\sim}$ is the set $\lbrace\langle\bar{x},i.A\rangle\mid\langle\bar{y},i,A\rangle\in S\mbox{ and }\bar{x}\sim\bar{y}\rangle\rbrace$.
\end{mdef}

\begin{lemma}
    Let $\mathcal{B}$ be an open branch in a saturated, complete tableau in the system including the Faithfulness Rule. Then $\mathcal{B}^{\sim}$ includes no inconsistent pairs of triples.
\end{lemma}

\begin{proof}
We prove this by contraposition. Suppose that there are triples $\langle\bar{x},1,C\rangle$ and $\langle\bar{x},0,C\rangle$ in $\mathcal{B}^{\sim}$. Then as $\bar{x}$ is a finite string, these tuples entered the closure of $\mathcal{B}$ by finitely many occasions of eliding substrings $nn$ given initial tuples $\langle\bar{y},1,C\rangle$ and $\langle\bar{y},0,C\rangle$ from the branch $\mathcal{B}$. But  by completeness of the branch, corresponding applications of the Faithfulness Rule must have led to  $\langle\bar{x},1,C\rangle$ and $\langle\bar{x},0,C\rangle$ having appeared in $\mathcal{B}$ itself, ruling out either its openness.
\end{proof}

\noindent Together, the foregoing lemmas set up a proof of the completeness of the tableau system with the Faithfulness Rule with respect to $\LRCN^{\sim}$:

\begin{theorem}[Completeness]
    If $X\nvdash^{\sim} A$, then $X\nvDash_{\LRCN^{\sim}} A$.
\end{theorem}

\begin{proof}
   Suppose that $X\nvdash^{\sim}A$ and find an open branch and take its faithful closure $\mathcal{B}^{\sim}$. As before, we can a valuation $f_{\mathcal{B}^{\sim}}(\bar{x},p)=1$ precisely when $\langle\bar{x},1,p\rangle\in\mathcal{B}^{\sim}$; by similar reasoning as used in the proof of completeness of $\mathbf{CLV}$, $f_{\mathcal{B}^{\sim}}\nVdash X\Yright A$. What remains to be shown is that we can find a \emph{faithful} $f_{\mathcal{B}^{\sim}}'$ witnessing this fact. By Lemma \ref{lemma:find-a-faithful-counterexample}, however, we are guaranteed that one exists, whence $X\nvDash_{\LRCN^{\sim}}A$.
\end{proof}

\section{Conclusion}\label{Section:Conclusion}

We introduced a novel and very aggressive form of hyperformality, \textit{lericone-hyperformality}. We showed that two well-known relevant logics are in fact lericone-hyperformal. As usual this gave us lericone variable-sharing as a corollary. Consequently, variable sharing  for $\logic{BM} $ and $\logic{B}$ can be freed from parasitism on variable sharing results for $\mathbf{R}$. 

We then introduced a new relevant logic that was also lericone-hyperformal---in fact, what we introduced was the largest lericone-hyperformal sublogic of classical logic---and a ``faithful'' sister system. These logics were defined semantically via the mechanism of lericone-sensitive assignments. We also provided proof systems for these logics in the form of tableau systems. 

There are many natural extensions of this work to be considered. Most pressing to us is actually giving a decent axiomatization of $\CLV$ and $\CLV^{\sim}$, and perhaps providing an alternative proof system for $\models_\LRCN$. It's also not clear how far we might extend the proof of our variable sharing result for $\CLV$. It doesn't work when tracking depth alone a la Brady~\citeyearpar{brady1984}---the largest sublogic of classical logic closed under depth substitutions will still contain all substitutions of $p\to(q\lor\neg q)$, so it doesn't enjoy variable sharing. But perhaps signed depth (a la Logan~\citeyearpar{logan2021}) would do the job. We're not sure. It would be good to know whether there is a single strongest lericone-relevant sublogic of \logic{CL} or whether there are multiple, incomparable such logics, as is the case with sublogics of \logic{CL} satisfying the variable-sharing criterion.\footnote{For the latter point, see \citet{Swirydowicz1999-SWITEE} and \citet{Mendez2012-MNDTEP}.
}

Finally, extensions of the vocabulary being considered naturally present opportunities for doing more syntax-tracking. Both modal extensions and extensions with constants in the $\mathsf{t},\mathsf{f},\top,\bot$ family are interesting in this regard.

\section*{Acknowledgments}
To be added. 

\bibliography{./bib}
\bibliographystyle{apalike}

\end{document}